\documentclass[10pt]{amsart}
\usepackage[T1]{fontenc}
\usepackage[utf8]{inputenc}
\usepackage[english]{babel}
\usepackage{csquotes} 
\usepackage[a4paper,margin=3cm]{geometry}
\allowdisplaybreaks
\usepackage{graphicx}
\usepackage[backend=biber,style=alphabetic,bibencoding=utf8,language=auto,autolang=other,giveninits=true,doi=false,isbn=false,url=false,maxnames=10,sorting=nty]{biblatex}
\usepackage{caption} 
\usepackage[hyperfootnotes=false]{hyperref} 
\hypersetup{
	colorlinks = true,
	linkcolor = {blue},
	urlcolor = {red},
	citecolor = {blue}
}
\usepackage[nameinlink,capitalise,sort]{cleveref}
\crefname{equation}{}{} 
\crefname{enumi}{}{} 

\usepackage{amsmath}
\usepackage{amsthm}
\usepackage{amssymb}
\usepackage{esint} 
\usepackage{mathrsfs} 
\usepackage{faktor} 
\usepackage{dsfont} 
\usepackage[dvipsnames]{xcolor}
\usepackage{array}
\usepackage{hhline}
\usepackage[normalem]{ulem}
\usepackage{enumitem} 
\usepackage{comment} 
\usepackage[toc,page]{appendix} 
\usepackage{imakeidx} 
\usepackage{xparse} 
\usepackage{mathtools}
\usepackage{fancyhdr} 
\usepackage{ifthen} 
\usepackage{forloop} 
\usepackage{xstring}
\usepackage{tikz}
\usetikzlibrary{babel} 
\usetikzlibrary{cd} 
\usepackage{emptypage} 
\usepackage{listings} 
\usepackage{mathabx} 

\theoremstyle{plain}
\newtheorem{lemma}{Lemma}[section]
\newtheorem{proposition}[lemma]{Proposition}
\newtheorem{theorem}[lemma]{Theorem}
\newtheorem{corollary}[lemma]{Corollary}

\theoremstyle{definition}
\newtheorem{definition}[lemma]{Definition}

\theoremstyle{remark}
\newtheorem{remark}[lemma]{Remark}

\numberwithin{equation}{section}

\newcommand{\bmo}{\mathrm{BMO}}

\newcommand{\loc}{\mathrm{loc}}

\newcommand{\R}{\mathbb{R}}
\newcommand{\N}{\mathbb{N}}

\newcommand{\eps}{{\varepsilon}}

\DeclareMathOperator{\supp}{supp}
\DeclareMathOperator{\dive}{div}
\DeclareMathOperator{\curl}{curl}

\DeclareMathOperator*{\esssup}{ess\,\sup}

\newcommand{\pt}{\partial_t}

\newcommand{\abs}[1]{\left\lvert#1\right\rvert}
\newcommand{\norm}[1]{\left\lVert#1\right\rVert}
\newcommand\skal[1]{\left\langle#1\right\rangle}

\renewcommand{\d}{\mathrm{d}}
\renewcommand{\div}{\operatorname{div}}

\newcommand{\dd}{\,\mathrm{d}}

\allowdisplaybreaks

\newcounter{eq}[section]

\newcommand{\mynorm}{{\vert\kern-0.25ex\vert\kern-0.25ex\vert}}
\usepackage{comment}

\newcommand{\ie}{{\itshape i.e.}}
\newcommand{\eg}{{\itshape e.g.}}

\newcommand\blank{{{}\cdot{}}}

\begin{document}

\title[]{Regularity aspects of Leray--Hopf solutions to the 2D inhomogeneous Navier--Stokes system and applications to weak-strong uniqueness}

\author[T.~Crin-Barat]{Timoth\'ee~Crin-Barat} 
\address[T.~Crin-Barat]{Université Paul Sabatier,  Institut de Math\'ematiques de Toulouse, Route de Narbonne 118, 31062 Toulouse Cedex 9, France.}
\email{timothee.crin-barat@math.univ-toulouse.fr}

\author[N.~De~Nitti]{Nicola~De~Nitti}
\address[N.~De~Nitti]{EPFL, Institut de Mathématiques, Station 8, 1015 Lausanne, Switzerland.}
\email{nicola.denitti@epfl.ch}

\author[S.~\v{S}kondri\'c]{Stefan~\v{S}kondri\'c} 
\address[S.~\v{S}kondri\'c]{Friedrich-Alexander-Universit\"at Erlangen-N\"urnberg, Department of Mathematics, Chair of Analysis, Cauerstra{\ss}e 11, 91058 Erlangen, Germany.}
\email{stefan.skondric@fau.de}

\author[A.~Violini]{Alessandro~Violini}
\address[A.~Violini]{Universit\"at Basel,  Department Mathematik und Informatik, Spiegelgasse 1, 4051 Basel, Switzerland. }
\email{alessandro.violini@unibas.ch}

	

\begin{abstract}
We characterize the Leray--Hopf solutions of the 2D inhomogeneous Navier--Stokes system that become strong for positive times. This characterization relies on the strong energy inequality and the regularity properties of the pressure. As an application, we establish a weak–strong uniqueness result and provide a unified framework for several recent advances in the field. 
\end{abstract}

\maketitle

\section{Introduction}
\label{sec:intro}

\subsection{Presentation of the problem}
\label{ssec:problem}

We consider the \emph{inhomogeneous incompressible Navier--Stokes system}, which describes incompressible flows with nonconstant densities (see~\cite[Chapter 1]{LionsVol11996}):
\begin{align}\label{eq:ns}
\begin{cases}
\pt (\rho u) +\div (\rho u \otimes u)-\nu \Delta u+\nabla P=0, &  (t,x) \in (0,\infty) \times \R^d, \\
\pt \rho +\div (\rho u )=0, & (t,x) \in (0,\infty) \times \R^d, \\
\operatorname{div} u=0, & (t,x) \in (0,\infty) \times \R^d, \\
\rho(0,x) = \rho_0(x), & x \in \R^d,\\
u(0,x) = u_0(x), & x \in \R^d,
\end{cases}
\end{align} 
where $\rho \geq 0 $ is the \emph{density function}, $P$ is the \emph{pressure}, $u$ is the \emph{velocity field} and $\nu$ is the \emph{diffusivity} of the fluid. System \cref{eq:ns} satisfies the following scaling invariance: if $(\rho, u, P)$ is a solution of \cref{eq:ns} with initial data $\left(\rho_0, u_0\right)$ then, for all $\lambda>0$, the rescaled triplet 
\begin{align}\label{eq:scalingnse}
\left(\rho_\lambda, u_\lambda, P_\lambda\right) \coloneqq \left(\rho\left(\lambda^2 t, \lambda x\right), \lambda u\left(\lambda^2 t, \lambda x\right),  \lambda^2 P\left(\lambda^2 t, \lambda x\right)\right)
\end{align}
is a solution of \cref{eq:ns} with initial data $(\rho_0(\lambda \blank), \lambda u_0(\lambda \blank))$. Furthermore, provided that a solution $(\rho, u, P)$ is sufficiently regular, the solution also conserves the energy: for every $t \in (0,\infty)$
\begin{align*}
    \frac{1}{2} \int_{\R^d} \rho(t)\abs{u(t,x)}^2 \dd x + \nu \int_0^t \int_{\R^2} \abs{\nabla u(s,x)}^2 \dd x \dd s = \frac{1}{2} \int_{\R^d} \rho_0(x)\abs{u_0(x)}^2 \dd x.
\end{align*}
Scale invariance suggests that the critical requirements for well-posedness are
$\rho_0 \in L^{\infty}(\mathbb{R}^d)$ and $u_0 \in \dot{H}^{\frac{d}{2}-1}(\mathbb{R}^d)$.
In this work, we focus on the two-dimensional case ($d=2$) with a strictly positive, bounded initial density and an $L^2$ divergence-free initial velocity:
\begin{equation}\label{ass:data}
\begin{aligned} 
    & 0 < c_0 \leq \rho_0(x) \leq C_0 < \infty \quad \text{for a.e.~$x \in \R^2$},\\
    &u_0 \in  L^2_\sigma(\R^2).
\end{aligned}
\end{equation}
For $d=2$, in \cite{Kazhikhov1974}, Kazhikhov established the existence Leray--Hopf weak solution for initial data \cref{ass:data} and, in~\cite{Simon1990}, Simon extended this result to the case $\rho_0 \ge 0$ (\ie, allowing for possible vacuum in the initial density).\footnote{~In \cite{LionsVol11996}, Lions demonstrated that the density is a renormalized solution of the mass equation and considered the cases of density-dependent viscosity coefficients and unbounded densities as well.}  
The uniqueness of Leray--Hopf weak solution of \cref{eq:ns} remains an open question (while, in the constant density case, for $d=2$, it  was established by Ladyzhenskaya, Lions, and Prodi (see~\cite{Ladyzhenskaya1959,LionsProdi1959}). However, well-posedness results were proven assuming additional regularity on the initial data. In~\cite{Danchin2003,Danchin2004}, Danchin proved that, if $u_0 \in \dot{B}_{2,1}^{0}(\R^2)$, then \cref{eq:ns} is globally well-posed provided $\rho_0$ is close to some positive constant in the homogeneous Besov space $\dot{B}_{2,1}^{1}(\R^2)$. In~\cite{AbidiGui2021}, Abidi and Gui  proved a similar result for $u_0 \in \dot{B}_{2,1}^0\left(\R^2\right)$ and without the smallness condition but requiring  $\rho_0^{-1}-1 \in \dot{B}_{2\slash \varepsilon,1}^\varepsilon(\R^2)$ for $\varepsilon \in (0,1).$ For merely bounded initial density, bounded away from zero, well-posedness results were first obtained by Danchin and Mucha in~\cite{DanchinMucha2013} and then improved by Paicu, Zhang, and Zhang in~\cite{PaicuZhangZhang2013}, where $u_0\in H^s\left(\mathbb{R}^2\right)$ for some $s>0$. Recently, in~\cite{HaoShaoWeiZhang2024}, Hao, Shao, Wei, and Zhang discussed the existence and uniqueness of solutions arising from initial data satisfying \cref{ass:data}. Additionally, in~\cite{AdogboMuchaSzlenk2024}, Adogbo, Mucha, and Szlenk constructed unique strong solutions for unbounded densities lying within a subspace of BMO.

We recall that a Leray--Hopf weak solution with initial data satisfying \cref{ass:data} is a distributional solution of \cref{eq:ns} that is energy admissible, meaning it satisfies the following energy inequality (see \cref{def:lerayhopfsol}):
\begin{align}\label{eq:enInequality}
    \frac{1}{2}\int_{\R^2} \rho(t) |u(t)|^2 \,\dd x + \nu \int_0^t \int_{\R^2} |\nabla u(s)|^2 
    \,\dd x \,\dd s \leq  \frac{1}{2}\int_{\R^2} \rho_0 |u_0|^2 \,\dd x \quad 
    \text{for every $t \in (0,\infty)$}.
\end{align}
Compared to the homogeneous model, several problems arise in the inhomogeneous case when analyzing Leray--Hopf weak solutions: 
\begin{itemize}[left=1pt]
\item No control of the pressure is guaranteed for Leray--Hopf weak solutions to \cref{eq:ns}: it is not even clear whether the pressure is a distribution. Unlike the homogeneous case, the regularity of the pressure cannot be recovered using Calderón–Zygmund-type estimates (as done in \cite[Theorem 1.13]{BedrossianVicol2022}). Indeed, the pressure $P$, at least for smooth triplet $(\rho,u,P)$ satisfying \cref{eq:ns}, solves
\begin{align}\label{eq:pressurerelation1}
    -  \div \left( \frac{\nabla P}{\rho} \right) = \div \div (u \otimes u) - \div\left( \frac{\nu \Delta u}{\rho}\right),
\end{align}
or, alternatively, the pressure can be recovered from the relation  
\begin{align}\label{eq:pressurerelation2}
    -  \Delta P = \div \div ( \rho u \otimes u) + \div\left( \rho \partial_t u \right).
\end{align}
In both \eqref{eq:pressurerelation1} and \eqref{eq:pressurerelation2}, the interaction between $P$ and $\partial_t u$ or $\Delta u$ makes it difficult to deduce any kind of regularity for $P$ when $\rho$ is merely bounded.

\item Secondly, it remains challenging to establish a strong version of the energy inequality \eqref{eq:enInequality}. In the homogeneous case, applying Ladyzhenskaya's inequality,\footnote{~That is, Gagliardo--Nirenberg's interpolation inequality \[ \|D^{j}u\|_{L^{p}(\mathbb {R} ^{d})}\leq C\|D^{m}u\|_{L^{r}(\mathbb {R} ^{d})}^{\theta }\|u\|_{L^{q}(\mathbb {R} ^{d})}^{1-\theta },\] with $d=2$, $j=0$, $m=1$, $p = 4$, $q = 2$, $r=2$,  and $\theta = 1/2$.\label{ft:gn}} we deduce 
    \[
        u \in L^4((0,\infty) \times \mathbb{R}^2),
    \]
    and conclude that \( u \) satisfies the energy equality (see \cite[Theorem 4.1]{Galdi2000}). The strategy is to test the weak formulation of the momentum equation using a spatial mollification of the velocity.   However, in the inhomogeneous case, this approach cannot be implemented as it is unclear whether \( u \), or any spatially mollified regularization of \( u \), can be used as a test function in the weak formulation of \cref{eq:ns}.\footnote{~The main difficulty here stems from the lack of sufficient time regularity for the velocity field \( u \). This issue arises because \( u \) is coupled with the density \( \rho \), preventing effective use of the divergence-free condition.} As a result, it remains an open question whether the same conclusion holds for non-constant densities.  In particular, it is not even known if every Leray--Hopf solution \( (\rho, u) \) satisfies the strong energy inequality: for almost every \( s > 0 \) and every \( t > s \),
    \begin{align} \label{SEI}
        \frac{1}{2} \int_{\R^2} \rho(t) \abs{u(t)}^2 \,\dd x
        + \nu \int_s^t \int_{\R^2} \abs{\nabla u(\tau)}^2 \,\dd x \,\dd \tau
        \leq \frac{1}{2} \int_{\R^2} \rho(s) \abs{u(s)}^2 \,\dd x.
    \end{align}

\item Due to the lack of a strong energy inequality, it is unclear whether there are any parabolic regularization effects. Given $u_0 \in H^k(\R^2)$, with $k \in \N$, there exists a strong solution $u$ of the homogeneous Navier--Stokes equations such that 
\begin{align*}
    u \in L^\infty((0,\infty); H^k(\R^2)) \cap L^2((0,\infty); \dot{H}^{k+1}(\R^2))
\end{align*}
(see \cite[Theorem 7.1]{RobinsonRodrigoSadowski2016}). In addition, for a Leray--Hopf solution, we have that $u(t) \in H^1(\R^2)$ for almost every $t > 0$. This implies the existence of a strong solution $\Tilde{u}_t$ with initial data $u(t) \in H^1(\R^2)$. By uniqueness, $\widetilde{u}_t$ coincides with a $t$-time translation of $u$, leading to a gain in regularity for $u$ for positive times:
\begin{align*}
    u \in L^\infty((t,\infty); H^1(\R^2)) \cap L^2((t,\infty); \dot{H}^{2}(\R^2)).
\end{align*}
 By the freedom in the choice of $t$ and iterating this procedure, we ultimately obtain 
\begin{align*}
    u \in C^\infty((0,\infty) \times \R^2)
\end{align*}
(see \cite[Theorem 7.5]{RobinsonRodrigoSadowski2016}). However, this result cannot be extended to Leray--Hopf solutions $(\rho,u)$ of \cref{eq:ns} with initial data satisfying \eqref{ass:data}. Although it is true that
\begin{align*}
    u(t) \in H^1(\R^2) \quad \text{for almost every $t > 0$,}
\end{align*}
and that there exists (owing to \cite{PaicuZhangZhang2013}) a strong solution $(\widetilde{\rho}_t, \widetilde{u}_t)$ with initial data $(\rho(t), u(t))$ such that
\begin{align}\label{eq:PZZsolintro}
    \widetilde{u}_t \in L^\infty((0,\infty); H^1(\R^2)) \cap L^2((0,\infty); \dot{H}^2(\R^2)),
\end{align}
it is not clear how to conclude that $(\widetilde{\rho}_t, \widetilde{u}_t)$ is equal to a $t$-time translation of $(\rho, u)$.  The main obstacle is that the weak-strong uniqueness argument (as presented, \eg, in \cite[Theorem 1.6]{CrinBaratSkondricViolini2024}) requires that $(\rho, u)$ is a Leray--Hopf solution on $[t,\infty) \times \R^2$. In particular, we need the energy inequality for $(\rho,u)$ in $[t,\infty)$ and, as already discussed, this is not guaranteed for every Leray--Hopf weak solution.

\end{itemize}

In our main result, \cref{thm:main2D}, we provide a connection between these three aspects.

\subsection{Main results}
\label{ssec:main}

In this paper, we introduce the class of \textit{immediately strong solutions} which is the particular subclass of Leray--Hopf weak solutions that become strong for any positive time. Our goal is to characterize this class and derive a panoply of weak-strong uniqueness results. We say that a Leray--Hopf solution $(\rho, u)$ is an \textit{immediately strong solution}  to \cref{eq:ns} with initial data in \cref{ass:data}  if
\begin{align}\label{eq:immBound}
    \partial_t u, \, \nabla^2 u \in L^2((\eps, \infty);\, L^2(\R^2)) \qquad \text{for every $\eps >0$}
\end{align}
(see \cref{def:Immstrongsol}).
In our first result, we establish a connection between the regularity properties, energy behavior, the existence of a suitable pressure, and the decay at infinity for this class of solutions.

\begin{theorem} \label{thm:main2D}
    Let us assume that the initial data $(\rho_0,u_0)$ satisfy \cref{ass:data} and let $(\rho,u)$ be a Leray--Hopf solution of \cref{eq:ns}. The following statements are equivalent:
    \begin{enumerate}[label=(\roman*)]
       \item \label{it:immediately} $(\rho,u)$ is immediately strong (in the sense of \cref{def:Immstrongsol});
        \item \label{it:strong} $(\rho,u)$ satisfies the strong energy inequality (in the sense of \cref{def:energies});
        \item  \label{it:A1A2A3} There exists a constant 
    $C=C(\nu,\norm{u_0}_{L^2},c_0,C_0)$ such that 
    \begin{align*}
        A_1^0(u), \, A_2^0(u), \, A_3^0(u) \leq C,
    \end{align*}
    where $A^0_i$, for $i \in \{1,2,3\}$, is defined in \cref{def:IS};
      \item \label{it:pressure} There exists an associated pressure $P$ (in the sense of \cref{def:assPress}) such that
        \begin{align*}
                   P \in L^2_\loc((\varepsilon,+\infty);\,\bmo(\R^2)) \cap L^2_\loc((\varepsilon, +\infty) \times \R^2) \qquad \text{for every $\varepsilon >0$}.
        \end{align*}
    \end{enumerate}
\end{theorem}
\begin{remark}[{BMO regularity for the pressure}]
The pressure gradient of strong solutions belongs to \( L^2(\R^2) \), and by the critical Sobolev embedding, see \cref{lem:embbmo}, this ensures that the pressure is in \( \mathrm{BMO}(\R^2) \). Theorem \ref{thm:main2D} establishes the converse: if the pressure is in \( \mathrm{BMO}(\R^2) \), then the pressure gradient is in \( L^2(\R^2) \), which makes it a natural regularity condition.
\end{remark}

We then turn to the question of the existence of solutions satisfying one of the equivalent conditions in \cref{thm:main2D}. Every Leray--Hopf weak solution produced by a smooth approximation argument is an immediately strong solution. In particular, for all initial data satisfying \cref{ass:data}, there exists an immediately strong solution of \cref{eq:ns}.

\begin{proposition}[Existence of immediately strong weak solution]\label{lem:exist}
    Let us assume that $(\rho_0,u_0)$ satisfies \cref{ass:data}. Then there exists an immediately strong solution $(\rho,u)$ of \cref{eq:ns}. 
\end{proposition}
 \cref{lem:exist} follows directly from the existence result established in \cite[Theorem 1.1]{Danchin2024}: indeed, the bounds derived therein imply \eqref{eq:immBound}. 
 
 \begin{remark}
    The existence of solutions that fail to satisfy one (and, therefore, all) of the properties mentioned in \cref{thm:main2D} remains an open question.
 \end{remark}

Finally, as a corollary of \cref{thm:main2D}, we observe that $u$ actually satisfies the energy equality, provided that the pressure satisfies the regularity assumption in \cref{it:pressure} up to time $t=0$.

\begin{corollary}[Energy equality]\label{cor:equality}
    Let us assume that the initial data $(\rho_0,u_0)$ satisfy \cref{ass:data} and let 
    $(\rho,u)$ be a Leray--Hopf solution of \cref{eq:ns}. Let us assume, in addition, that 
    \begin{align}
        P \in L^2_\loc([0,+\infty); \, \bmo(\mathbb{R}^2)) \cap L^2_\loc([0, +\infty) \times \mathbb{R}^2).
    \end{align}
    Then $u$ satisfies the energy equality.
\end{corollary}
This corollary take inspiration from the result \cite[Theorem 1.1]{BieWang2023}, where Wang and Bie prove the energy equality for weak solutions of the inhomogeneous incompressible Hall--MHD system in a bounded domain. Here, combining their method with the additional $\bmo$-bound on the pressure allows us deal with the $[0,\infty) \times \R^2$ setting for the incompressible inhomogeneous Navier--Stokes equations.
\subsubsection{Application to weak-strong uniqueness}
\label{ssec:applications}

As mentioned above, in the case of constant densities, the uniqueness of Leray--Hopf solutions in two spatial dimensions has been known since the works by Ladyzhenskaya, Lions, and Prodi (see~\cite{Ladyzhenskaya1959,LionsProdi1959}). However, this remains an open problem in the inhomogeneous framework. One of the primary reasons for this difficulty is the following observation. Let $(\rho_1,u_1,P_1)$ and $(\rho_2,u_2,P_2)$ be two sufficiently regular solutions of \cref{eq:ns} with respect to the same initial data $(\rho_0,u_0).$ A simple calculation shows that, for every $t >0$, 
\begin{align}\label{eq:energydiffintro}
    \frac{1}{2}\norm{\sqrt{\rho_1(t)} \delta u(t)}_{L^2}^2 + \nu \int_0^t 
    \norm{\nabla \delta u}^2_{L^2} \dd s  &= -\int_0^t \int_{\R^2} \delta \rho 
    \dot{ u_2} \cdot \delta u -\int_0^t \int_{\R^2}  \rho_1 \delta u \otimes \delta u : 
    \nabla u_2,
\end{align}
where we used the notations
\begin{align*}
    \dot{u}_2 \coloneqq \partial_t u_2 + (u_2 \cdot \nabla) u_2, \quad \delta \rho \coloneqq \rho_1 - \rho_2, \quad \text{and  }\: \delta u \coloneqq u_1 - u_2.
\end{align*}
In the case of constant density, the first term on the right-hand side of \eqref{eq:energydiffintro} does not appear, and it is precisely this term that introduces the main technical difficulties. In order to get uniqueness, one usually relies on Grönwall's inequality. To this end, it would be necessary find a function $f \in L^1_\loc([0,\infty))$ such that the right-hand side of \eqref{eq:energydiffintro} is bounded by
\begin{align}\label{eq:f}
    \int_0^t f(s)\norm{\sqrt{\rho_1(s)} \delta u(s)}_{L^2}^2 \dd s + \frac{\nu}{2} \int_0^t 
    \norm{\nabla \delta u}^2_{L^2} \dd s.
\end{align}
Then, by Grönwall's inequality, we would deduce that $\delta u \equiv 0$, which implies that
\begin{align*}
    \partial_t \delta \rho + \div(u_2 \delta \rho ) = 0,\quad \delta \rho(0) = 0.
\end{align*}
Hence, by assuming, for instance, that $u_2$ is Lipschitz continuous in space, we can conclude that $\delta \rho (t)=0$ for almost every $t>0$.

However, the first term on the right-hand side of \eqref{eq:energydiffintro} is not quadratic in $\delta u$ and it is therefore unclear how to rewrite it in the form \eqref{eq:f}. To capture the quadratic structure in terms of $\delta u$, one way is to perform a stability analysis of the continuity equation 
\begin{align}\label{eq:deltarhointro}
    \partial_t \delta \rho + \div (\delta \rho u_2) = \div(\rho_1 \delta u) = - \nabla \rho_1 \cdot \delta u,
\end{align}
so that the $\delta \rho$ term can be bounded in term of $\delta u$. In doing so, the right-hand side of \eqref{eq:deltarhointro} causes a loss of one derivative, and, as discussed in \cite[Section 1.4]{CrinBaratSkondricViolini2024}, this requires estimating higher derivatives of the material derivative \( \dot{u}_2 \). Indeed, following this procedure, we obtain 
\begin{align}\label{eq:productestimateintro}
\begin{aligned}
    - \int_0^t \int_{\R^2} \delta \rho \dot{u}_2 \cdot \delta u \dd x \dd s &\lesssim \int_0^t \int_0^s \int_{\R^2}(\dive\left(\rho_1 \delta u))(\tau,X(\tau,x)
    \right)\dot{u}_2\cdot \delta u  \dd\tau \dd x\dd s
    \\&\lesssim 
     \int_0^t \int_0^s \norm{(\rho_1 \delta u)(\tau)}_{L^4} \dd \tau \norm{\nabla(\dot{u}_2(s) \cdot \delta u(s))}_{L^{4\slash 3}} \dd s, \end{aligned}
\end{align} where $X$ is the flow associated with the continuity equation \eqref{eq:deltarhointro}; see \cite{CrinBaratSkondricViolini2024} for more details.
If \((\rho_2, u_2)\) is an immediately strong solution, the higher derivative estimates for the material derivative \( \dot{u}_2 \) are ensured by \cref{it:A1A2A3} and we are able to conclude the uniqueness argument with Gronwäll's inequality provided that $u_2$ is continuous in time at \( t = 0 \). In this spirit, we establish the following weak-strong uniqueness result.

\begin{theorem}[Critical weak-strong uniqueness]\label{thm:weakstrong2D}
    Let us suppose that the initial data $(\rho_0,u_0)$ satisfy \cref{ass:data}, and let $(\rho_1,u_1)$ and $(\rho_2,u_2)$ be two Leray--Hopf solutions of \cref{eq:ns}. Let us assume, moreover, that $(\rho_2,u_2)$ is immediately strong (or one of the other
    equivalent conditions stated in  \cref{thm:main2D}) and that 
    \begin{align}\label{ass:data-thu}
        \partial_t u_2 \in L^1_\loc([0,\infty);L^2(\R^2)), 
        \quad \nabla u_2 \in L^1_\loc([0,\infty);L^\infty(\R^2)).
    \end{align}
    Then, 
    \begin{align*}
        (\rho_1,u_1) = (\rho_2,u_2) \quad \textrm{a.e in $[0,\infty)\times \R^2$}.
    \end{align*}
\end{theorem}
\begin{remark}

    All the quantities appearing in \cref{thm:weakstrong2D} are critical with respect to the scaling in \cref{eq:scalingnse}, which is an improvement compared to the weak-strong uniqueness results in \cite[Theorem 1.6]{CrinBaratSkondricViolini2024}. 
    \end{remark}
   \begin{remark}
   To motivate the assumptions \cref{ass:data-thu}, we recall that, in \cite{Danchin2024}, Danchin proved that two solutions \( (\rho_1, u_1, P_1) \) and \( (\rho_2, u_2, P_2) \) of \cref{eq:ns} coincide if they satisfy the following conditions:
\begin{align*}
   \quad \quad \quad \partial_t u_i \in L^1_\loc([0,\infty);L^2(\R^2)),\quad
    \nabla u_i \in L^1((0,\infty); L^\infty(\R^2)), \quad 
    A_1^0( u_i), \, A_2^0( u_i), \, A_3^0(u_i) \leq C,
\end{align*}
for $i\in \{1,2\}$, where $A^0_j$, for $j \in \{1,2,3\}$, is defined in \cref{def:IS}. In \cref{thm:weakstrong2D}, the regularity assumptions are required only for one of the two solutions extending the uniqueness result in \cite{Danchin2024}. This improvement stems from avoiding the transition to Lagrangian coordinates by employing the relative energy method, which allows for relaxed regularity assumptions on one of the two solutions.

\end{remark}

Numerous classes of strong solutions studied in the literature verify the hypothesis imposed on $(\rho_2,u_2)$ in \cref{thm:weakstrong2D}. Thus, as a direct corollary of \cref{thm:weakstrong2D}, we deduce that the strong solutions constructed in~\cite[Theorem 1.1]{AbidiGuiZhang2023},~\cite[Theorem 1.4]{DanchinWang2023}, and~\cite[Theorem 1.3]{Danchin2024} are unique in the class of Leray--Hopf solutions.

\begin{corollary}[Uniqueness]\label{cor:uniqueness-ll}
Let us suppose that the initial data $(\rho_0,u_0)$ satisfy \cref{ass:data}. Let us assume, in addition, that one of the following conditions holds:
\begin{enumerate}
    \item $u_0\in \dot{B}^{-1+\frac2p}_{p,1}$ and $1-\rho_0^{-1}\in \dot{B}^{\frac2\lambda}_{\lambda,2}$, for $p\in[2,+\infty)$ and $\lambda\in[1,\infty)$ such that $1/2<1/p+1/\lambda\leq1$ (framework of~\cite{AbidiGuiZhang2023});
    \item $u_0\in \dot{B}^{-1+\frac2p}_{p,1}$,  $\|\rho_0-1\|_{L^\infty}$ small enough, with $p\in (1,2)$ (framework of~\cite{DanchinWang2023}); 
    \item $u_0\in \widetilde{B}^{0,s}_{\rho_0,s}$ for some $s\in(0,1)$ (framework of~\cite{Danchin2024}).
\end{enumerate}
Then \cref{eq:ns} admits a global-in-time strong solution that is unique in the class of Leray--Hopf solutions.
\end{corollary}

\begin{remark}
\cref{cor:uniqueness-ll} extends the uniqueness result of \cite[Theorem 1.1, Corollary 1.1]{AbidiGuiZhang2023} to a wider range of indices.
\end{remark}

\subsection{Strategy of proof of \texorpdfstring{\cref{thm:main2D}}{Theorem 1.1}}
\label{ssec:diff}

Let us outline the key elements of the strategy used in the proof of \cref{thm:main2D}.

Immediately strong solutions satisfy the energy equality on $(0, \infty)$, see \cref{lem:energyeq1}, and,  together with the energy inequality, this implies the strong energy inequality. This shows that  \cref{it:immediately} implies \cref{it:strongIn}. For the reverse implication, we demonstrate that the strong energy inequality grants a semi-flow property to the Leray--Hopf solutions; that is, any time translation of such a solution remains a Leray--Hopf solution with respect to the corresponding translated initial data, and, at the same time, the time translation is much more regular. More precisely, given a Leray--Hopf solution $(\rho,u)$ which satisfies the strong energy inequality, and a $t>0$ such that $u(t) \in H^1(\R^2)$, then, by \cref{thm:weakstrong2D}, the $t$-time translation of $(\rho,u)$ coincides with the more regular solution $(\widetilde{\rho}_t,\widetilde{u}_t)$, see \cref{eq:PZZsolintro}, and, by repeating this procedure for arbitrarily small times, we conclude that $(\rho,u)$ has the same regularity as $(\widetilde{\rho}_t,\widetilde{u}_t)$ for every $t>0.$ The crucial point is that the strong energy inequality guarantees that any $t$-time translation of $(\rho,u)$ is a Leray--Hopf solution with respect to the initial data $(\rho(t),u(t))$, which allows us to use the \cite[Theorem 1.6]{CrinBaratSkondricViolini2024}.

The implication from \cref{it:A1A2A3} to \cref{it:immediately} is a direct consequence of the definition of $A^0_1$. For the reverse implication, assuming both \cref{it:immediately} and \cref{it:strongIn}, we make use of the result in \cite[Theorem 1.1]{Danchin2024}. 

Let us describe how the equivalence between \cref{it:immediately}, \cref{it:strongIn}, \cref{it:A1A2A3} and \cref{it:pressure} can be established. To this end, let $(\rho,u,P)$ be a smooth triplet satisfying
\begin{align*}
    \nabla P = - \rho \dot{u} + \nu \Delta u \coloneqq G.
\end{align*}
Clearly, this condition implies that $\curl G = 0,$ and the (unique, up to a constant) potential, still denoted by $P$, can be computed by the formula
\begin{align}\label{eq:antigradientintro}
    G(t) \mapsto P(t) = \int_0^1 x \cdot G(t,\tau x) \dd \tau.
\end{align}
Note that, when assuming that $G$ is an immediately strong Leray--Hopf solution, we also have that, for every $t>0$, $G(t) \in L^2(\R^2) \cap L^4(\R^2)$. In \cref{app:grad}, we prove that the formula in \cref{eq:antigradientintro} can be extended to an operator
\begin{align}
    \Phi \colon L^2(\R^2) \cap L^4(\R^2) \to C^{1/2}_0(\R^2) \cap \bmo(\R^2),
\end{align}
which maps every smooth, curl-free vector field in $L^2(\R^2) \cap L^4(\R^2)$ to its normalized potential, i.e., $P(0) = 0.$ This shows that, when treating the time as parameter, the associated pressure to an immediately strong solution $(\rho,u)$ can be obtained by setting $P = \Phi G.$ Then, due to the properties of the operator $\Phi$, we have the desired bounds
\begin{align}\label{eq:pressureregintro}
    P \in L^2_\loc((0,\infty)\times \R^2) \cap L^2_\loc((0,\infty);\bmo(\R^2)).
\end{align}
This shows that \cref{it:A1A2A3} $\implies$ \cref{it:pressure}. Due to the equivalence of \cref{it:immediately}, \cref{it:strongIn}, and \cref{it:A1A2A3}, it suffices to show that \cref{it:pressure} implies \cref{it:strongIn} in order to complete the proof of \cref{thm:main2D}. 

This is done in two steps. First, inspired by \cite{BieWang2023}, we use a regularization procedure and commutator estimates to show that, under the assumption  \cref{eq:pressureregintro},  $(\rho,u,P)$ is a suitable solution: namely, for every $\varphi \in C^\infty_c((0,\infty)\times \R^2)$,
\begin{align}\label{eq:suitsolintro}
    \begin{aligned}
      &- \int_0^\infty\int_{\R^2} \frac{1}{2} \rho |u|^2 \partial_s \varphi  \, \mathrm d x \, \mathrm d s - \nu  
        \int_0^\infty\int_{\R^2} \frac{1}{2}|u|^2 \Delta \varphi \, \mathrm d x \, \mathrm d s \\&\quad\quad\quad\quad+ \nu 
        \int_0^\infty\int_{\R^2} |\nabla u|^2 \varphi \, \mathrm d x \, \mathrm d s  - \int_0^\infty\int_{\R^2} \left(\frac{1}{2} \rho |u|^2  u + P  u \right) \cdot \nabla \varphi \, \mathrm d x \, \mathrm d s = 0.
    \end{aligned}
\end{align}
This shows that a Leray--Hopf solution with a more regular pressure satisfies the local energy inequality. Unlike in the homogeneous case, where (owing to \cite{Kukavica2006}) the energy equality is true when just assuming $P \in L^2_\loc((0,\infty)\times \R^2)$, an additional assumption is needed to conclude the strong energy inequality. Then, a possible way to deduce the strong energy inequality from \cref{eq:suitsolintro} is to choose a sequence of test functions converging to $1$ which are independent of time. The additional bounds in $\bmo$ on the pressure and the bounds on $u$ ensure that all terms on the right-hand side of \cref{eq:suitsolintro} vanish in the limit.

\subsection{Outline of the paper} The paper is organized as follows. In \cref{sec:solutions}, we discuss the notion of solutions and their main properties. In \cref{sec:i-ii-iii}, we prove the equivalence between \cref{it:immediately}, \cref{it:strong}, and \cref{it:A1A2A3}. More precisely, in \cref{subsec:SEI}, we prove the semi-flow property of immediately strong solution and the equivalence between \cref{it:immediately} and \cref{it:strong}; in \cref{subse:Timedec}, we discuss the additional regularity properties given by \cref{it:A1A2A3} and prove that \cref{it:A1A2A3} implies \cref{it:strong} and that  \crefrange{it:immediately}{it:strong} imply \cref{it:A1A2A3}. In \cref{sec:pressure}, we study the pressure and prove the equivalence between \crefrange{it:A1A2A3}{it:A1A2A3} and \cref{it:pressure}. More precisely, in \cref{ssec:pr1}, we prove that \cref{it:pressure} implies \cref{it:strong}; in \cref{ssec:pr2}, we prove that \cref{it:immediately} and \cref{it:A1A2A3} imply \cref{it:pressure}. In \cref{ssec:pr1}, we also present the proof of \cref{cor:equality}. In \cref{sec:uniqueness}, we prove \cref{thm:weakstrong2D}. Finally, in \crefrange{app:conv-comm}{app:grad}, we collect several technical results needed throughout the paper. 

\subsection{Notations}

Throughout the paper, we use  $\mathscr{L}^d$ to denote the $d$-dimensional \textit{Lebesgue measure} and the standard notation for \textit{Lebesgue spaces}: for a domain $\Omega \subset \mathbb{R}^d$, we let  $L^p(\Omega)$, for $p \in [1,\infty]$, denote space of all measurable functions $f$ with finite norm 
$$
\begin{aligned}
& \|f\|_{L^p(\Omega)}\coloneqq \left(\int_{\Omega}|f(x)|^p\, \mathrm d x \right)^{1 / p}, \qquad  \text{if } p \in [1,\infty),  \\
& \|f\|_{L^{\infty}(\Omega)}\coloneqq \underset{x \in \Omega}{\operatorname{ess} \sup }|f(x)|\coloneqq \inf \{C:\, |f| \leq C \text { a.e.~in } \Omega\} .
\end{aligned}
$$
We also use the notation $L^p_\sigma(\Omega)$ to denote the space of $L^p$-functions with zero (distributional) divergence. We use $L^p_{\text{loc}}(\Omega)$ for the space of locally \( p \)-integrable, functions, 
\[
L_{\mathrm{loc}}^p(\Omega) = \left\{ f : \Omega \to \mathbb{R} \, \text{ measurable}: \,  \ f|_K \in L^p(K),  \text{ for all }  K \subset \Omega, \, K \text{ compact} \right\}.
\]
When the domain $\Omega$ is clear from the context, we omit it and write only $L^p$, $L^p_{\mathrm{loc}}$ or $L^p_\sigma$.

Given a Banach space $X$, we define the usual \textit{Lebesgue--Bochner space} $L^p([0, T]; X)$: we denote the usual  of strongly measurable maps $f : [0, T] \to X$ with 
\[
\|f\|_{L^p([0, T]; X)} \coloneqq \left( \int_{0}^T \|f\|_X^p \, \mathrm{d}t \right)^{1/p} < \infty.
\]
Similarly, \( C_w([0,T]; X) \) or $C_{w^*}([0,T]; X)$ denote the space of functions \( f : [0,T] \to X \) that are continuous with respect to the weak or weak-$\ast$ topology of \( X \), which means that, for all \( \phi \in X' \),
\(
t \mapsto \langle f(t), \phi \rangle
\)
is a continuous function on \( [0,T] \). We also adopt the symbols $C_{c}$ or $C_{c,\sigma}$ to denote continuous functions with compact support or zero (distributional) divergence.

We define by $f_\gamma$ the mollification of $f$ in space, \ie, the convolution of $f$, with respect to the space variable, with the standard mollifier $\eta_\gamma$.

For a set $M$ and functions $f,\, g: M \rightarrow \mathbb{R}_{+}$, we  write $f(m) \lesssim g(m)$ if there exists a constant $C>0$, independent of $m$, such that $f(m) \leq C g(m)$ for all $m \in M$ (and similarly for $\gtrsim$).
For \textit{Besov spaces}, we adopt the notation of \cite{BahouriCheminDanchin2011}.
Finally, we let $C$ denote a generic constant that might change from line to line.

\section{Notions of solutions}
\label{sec:solutions}

In this section, we recall the notions of solutions used throughout the manuscript.

\subsection{Leray--Hopf solutions}\label{ssec:lh}

First, we recall the definition of Leray--Hopf weak solutions. 

\begin{definition}[Leray--Hopf weak solution]\label{def:lerayhopfsol}
A pair $(\rho,u)$ is a \textit{Leray--Hopf} weak solution of \cref{eq:ns} with initial data $(\rho_0,u_0)$ 
satisfying \cref{ass:data} if
    \begin{itemize}
        \item[(i)] The solution satisfies
        \begin{align*}
           & \sqrt{\rho} u \in L^\infty((0,\infty);\,L^2(\R^2)),
            \quad \rho \in L^\infty((0,\infty)\times \R^2),
            \\ &  u \in L^2_{\loc}([0,\infty)\times\R^2)  \quad \text{and} \quad \nabla u \in L^2((0,\infty) \times \R^2); 
        \end{align*}
        \item[(ii)] The pair $(\rho,u)$ is a distributional solution of \cref{eq:ns}:
        \smallbreak
        \begin{itemize}
            \item[$\bullet$] For every  $\varphi \in C_c^\infty([0,\infty) \times \R^2; \R)$,
        \begin{align*}
             \int_0^\infty \int_{\R^2} \rho \partial_s \varphi \dd x \dd s
            +  \int_0^\infty \int_{\R^2} \rho u \cdot \nabla \varphi \dd x \dd s 
            =- \int_{\R^2} \rho_0 \varphi(0) \dd x;
        \end{align*}
        \item[$\bullet$] For every $\varphi \in C_{c,\sigma}^\infty([0,\infty) \times \R^2;\R^2)$;
        \begin{align*}
            \int_0^\infty \int_{\R^2} &\rho \partial_s \varphi \cdot u+
            \rho u \otimes u : \nabla \varphi \dd x \dd s  
            - \nu \int_0^\infty \int_{\R^2} \nabla u : \nabla \varphi \dd x \dd s = -\int_{\R^2} \rho_0 u_0 \cdot \varphi(0)\dd x;
        \end{align*}
        \item[$\bullet$]For every  $\varphi \in C_c^\infty([0,\infty) \times\R^2;\R)$,
        \begin{equation*}
            \int_0^\infty\int_{\R^2} u\cdot \nabla \varphi \dd x\dd s =0;
        \end{equation*}
  \end{itemize}
         \item[(iii)] The energy inequality holds:
        \begin{align*}
           \frac{1}{2} \int_{\R^2} \rho(t) |u(t)|^2 \dd x + \nu \int_0^t \int_{\R^2} |\nabla u(s)|^2 \dd x \dd s \leq 
           \frac{1}{2} \int_{\R^2} \rho_0|u_0|^2 \dd x \quad \text{for every $t \in (0,\infty)$}.
        \end{align*}
     
    \end{itemize}
\end{definition}
\begin{remark}
Leray--Hopf weak solutions starting from  a no-vacuum state do not develop vacuum (see~\cite[Theorem 1.8]{CrinBaratSkondricViolini2024}), \ie, if $(\rho,u)$ is a Leray--Hopf weak solution, then
    \begin{align*}
        0 < c_0 \leq \rho_0(x) \leq C_0 \quad \text{for~a.e.~$x \in \R^2$}  \implies  0 < c_0 \leq \rho(t,x) \leq C_0 \quad \text{for~a.e.~$x \in \R^2$ and every $t>0$}.
    \end{align*}
\end{remark}
\begin{definition}[Strong energy inequality and energy equality]\label{def:energies}
    We say that a Leray--Hopf solution $(\rho,u)$ of \cref{eq:ns}, with initial data $(\rho_0,u_0)$ satisfying  \cref{ass:data}, satisifies   
    \begin{enumerate}
        \item the \emph{strong energy inequality} if there exists a set $J \subset (0,\infty)$ such that $\mathscr{L}^1((0,\infty) \setminus J) = 0$ and, for every $t,\,s \in J$, with $t>s$:
    \begin{align*}
        \frac{1}{2} \int_{\R^2} \rho(t) |u(t)|^2 \dd x +\nu \int_s^t\int_{\R^2} |\nabla u(\tau)|^2 \dd x \dd \tau 
        \leq \frac{1}{2} \int_{\R^2} \rho(s) |u(s)|^2 \dd x;
    \end{align*}
    \item the \emph{energy equality} if, for every $t,s \in [0,+
    \infty)$, with $s<t$,
    \begin{align*}
        \frac{1}{2} \int_{\R^2} \rho(t) |u(t)|^2 \dd x + \nu\int_s^t\int_{\R^2} |\nabla u(\tau)|^2 \dd x \dd \tau 
        = \frac{1}{2} \int_{\R^2} \rho(s) |u(s)|^2 \dd x.
    \end{align*}
    \end{enumerate}
\end{definition}
\begin{remark}[Continuity up to the initial time]
   Our definition of \textit{strong energy inequality} is slightly weaker than the classical one, where the inequality is required for every \( s \in J \) and every \( t > s \). We will see that, for immediately strong solutions, the two definitions are equivalent. Moreover, for immediately strong solutions, the energy equality holds for every \( s, t \in (0, \infty) \) (see the proof of \cref{prop:contitimeSEI} and \cref{lem:StrongIneq}). 
\end{remark}

\subsection{Strong solutions} \label{ssec:strong}

We define a subclass of strong solutions, within the Leray--Hopf ones.

\begin{definition}[Strong solutions]\label{def:strongsol}
    We say that a Leray--Hopf solution \( (\rho, u) \) with initial data \( (\rho_0, u_0) \) satisfying \cref{ass:data} is a \textit{strong solution} to \cref{eq:ns} if 
    \begin{align}\label{eq:strongCond}
        \partial_t u, \, \nabla^2 u \in L^2((0, \infty);\, L^2(\R^2)).
    \end{align}
    Up to a modification on a negligible set of times, we can assume also $u\in C([0,\infty);L^2(\R^2))$.
\end{definition}

Let us compare the notion of strong solution we adopt in \cref{def:strongsol} with the solutions constructed in some classical well-posedness results from the literature. In~\cite{PaicuZhangZhang2013}, the authors build a solution from initial data $(\rho_0,u_0)$ in \cref{ass:data} with the additional assumption that $u_0 \in H^1$. Their solution satisfies the following regularity estimates:
\begin{align}\label{def:S}
\begin{aligned}
    & A_1(u) \coloneqq  \esssup_{s \in (0, \infty)} \int_{\R^2} |\nabla u|^2 \dd x + \int_0^\infty  \int_{\R^2} |\partial_s u|^2 + |\nabla^2 u|^2 + \abs{\nabla P}^2 \dd x \dd s \leq C,\\
    & A_2(u) \coloneqq \esssup_{s \in (0, \infty)} s \int_{\R^2}  |\partial_s u|^2 + |\nabla^2 u|^2 + \abs{\nabla P}^2 \dd x + \int_0^\infty s \int_{\R^2}  |\partial_s \nabla u|^2 \dd x \dd s \leq C.
    \end{aligned}
\end{align} 
\begin{remark}\label{rk:matDer}
    The regularity in  \cref{def:S} implies that\footnote{~The key idea is to use Agmon's inequality (that is, Gagliardo--Nirenberg's interpolation inequality 
    $ \|D^{j}u\|_{L^{p}(\mathbb {R} ^{d})}\leq C\|D^{m}u\|_{L^{r}(\mathbb {R} ^{d})}^{\theta }\|u\|_{L^{q}(\mathbb {R} ^{d})}^{1-\theta },$
    with $d=2$, $j=0$, $m=0$, $p = \infty$, $q = 2$, $r=2$,  and $\theta = 1/2$). We sketch the first estimate: \begin{align*}
     \int_0^\infty \norm{ u \cdot \nabla u}^2_{L^2} \dd s  \leq \int_0^\infty \norm{ u}^2_{L^\infty} \norm{\nabla u}^2_{L^2} \dd s  &\leq \int_0^\infty \norm{ u}_{L^2} \norm{\nabla^2 u}_{L^2} \norm{\nabla u}^2_{L^2} \dd s\\ &\leq \norm{ u}_{L^\infty L^2} \norm{ \nabla u}_{L^\infty L^2} \norm{ \nabla^2 u}_{L^2 L^2}.
\end{align*}}
   $u \cdot \nabla u \in L^2((0, \infty);\, L^2(\R^2))$. 
   Combining this with the time-regularity in \cref{eq:strongCond} also yields 
  $\dot{u} \in L^2((0, \infty);\, L^2(\R^2))$. In other words, a strong solution in the sense of \cite{PaicuZhangZhang2013} is a Leray--Hopf weak solution where every term of the momentum equation belongs to $L^2((0, \infty);\, L^2(\R^2))$.
\end{remark}
It is clear that the solution built in~\cite{PaicuZhangZhang2013} is a strong solution according to \cref{def:strongsol}, but it is not immediate to see that our notion of strong solution satisfies \cref{def:S}. To prove this fact (and thus that the two notions are equivalent), we need the following special case of~\cite[Theorem 1.6]{CrinBaratSkondricViolini2024}.

\begin{theorem}[Weak-strong uniqueness in 2D]\label{thm:weakstrong1}
        Let us assume that the initial data $(\rho_0,u_0)$ satisfies \cref{ass:data} and let $(\rho_1, u_1)$ be a {Leray--Hopf} weak 
    solution of \cref{eq:ns}. Let $(\rho_2, u_2)$ be a strong solution of \cref{eq:ns} such that \cref{def:S} holds and with the same initial data 
    $(\rho_0, u_0)$. Then we have $(\rho_2, u_2) = (\rho_1, u_1).$
\end{theorem}
We are now able to show that our notion of strong solutions is consistent with the one introduced in~\cite{PaicuZhangZhang2013}. Moreover, a solution which satisfies \cref{eq:strongCond} always arises from an $H^1$ initial velocity.
\begin{theorem}\label{thm:pzz}
     Let us assume that the initial data $(\rho_0,u_0)$ satisfy \cref{ass:data} and let $(\rho,u)$ be a Leray--Hopf solution of \cref{eq:ns}. The following two conditions are equivalent:
   \begin{enumerate}[label=(\roman*)]
        \item \label{it:strongC}$(\rho,u)$ is a strong solution;
        \item\label{it:strongIn} $u_0 \in H^1(\R^2)$.
    \end{enumerate}
Furthermore, if one of the two previous conditions holds, then the (unique) Leray--Hopf solution $(\rho, u)$ satisfies \cref{def:S} and, after the modification on a negligible set of times, we have that 
\begin{align*}
    u \in C([0,\infty);H^1(\R^2)).
\end{align*}
\end{theorem}

The proof of \cref{thm:pzz} relies on the use of \cref{thm:weakstrong1}. To do so, it is necessary to verify that strong solutions satisfy the energy equality. For the sake of completeness, we present this result in the following lemma.

\begin{lemma}\label{lem:energyeq1}
     Let us assume that the initial data $(\rho_0,u_0)$ satisfy \cref{ass:data} and let $(\rho,u)$ be a strong solution of \cref{eq:ns}. Then $(\rho,u)$ satisfies the energy equality.
\end{lemma}

\begin{proof}
   By \cite[Lemma 1.31]{RobinsonRodrigoSadowski2016}, we have that, for every $t \geq 0$, 
    \begin{align*}
        u(t) = u_0 + \int_0^t \partial_s u(s) \dd s,
    \end{align*}
    and, for every $\gamma > 0$,
    \begin{align}\label{eq:molliftimederi}
        u_\gamma(t) = (u_0)_\gamma + \int_0^t (\partial_s u(s))_\gamma \dd s.
    \end{align}
    From this we see that, for every $T>0$,
    \begin{align*}
        u_\gamma \in H^1((0,T);H^1(\R^2)) \cap L^2((0,T);H^2(\R^2)).
    \end{align*}
    Hence, $u_\gamma$ is an admissible test function in the weak formulation for $(\rho,u),$ see \cite[Lemma 3.3]{CrinBaratSkondricViolini2024}. Using first the weak formulation of the conservation of mass and then the weak formulation of the momentum equation, we obtain
    \begin{align*}
        \frac{1}{2} \int_{\R^d} &\rho(t) |u_\gamma(t)|^2 \dd x - \frac{1}{2} \int_{\R^d} \rho(0) |u_\gamma(0)|^2 \dd x 
      \\& = \int_0^t\int_{\R^d} \rho \partial_s u_\gamma \cdot u_\gamma + \rho u \otimes u_\gamma : 
        \nabla u_\gamma \dd x \dd s\\
        &= \int_0^t\int_{\R^d} \rho \partial_s u_\gamma \cdot (u_\gamma - u) + \rho u \otimes 
        (u_\gamma - u) : \nabla u_\gamma \dd x \dd s\\
        &\quad + \int_0^t \int_{\R^2} \rho \partial_s u_\gamma \cdot u 
        + \rho u \otimes u : \nabla u_\gamma \dd x \dd s\\
        &= \int_0^t\int_{\R^d} \rho \partial_s u_\gamma \cdot (u_\gamma - u) + \rho u \otimes 
        (u_\gamma - u) : \nabla u_\gamma \dd x \dd s\\
        &\quad + \nu \int_0^t \int_{\R^2} \nabla u : \nabla u_\gamma \dd x \dd s + 
        \int_{\R^d} \rho(t) u_\gamma(t) \cdot u(t) \dd x- \int_{\R^d} (u_0)_\gamma  u_0  \dd x.
    \end{align*}
    Employing the a priori bound on $u$ and \cref{eq:molliftimederi}, we get
    \begin{align*}
        u_\gamma \to u \qquad \text{in  $L^2_\loc((0,\infty);L^2(\R^2))$
         and  $L^4((0,\infty);L^4(\R^2))$} \quad \text{as $\gamma \to 0$}.
    \end{align*}
    Furthermore, since $(\partial_t u)_\gamma \to \partial_t u$ in $L^2((0,\infty);L^2(\R^2))$, we deduce that 
    \begin{align*}
        u_\gamma \to u \qquad \text{in $C_\loc([0,\infty);L^2(\R^2))$} \quad \text{as $\gamma \to 0$}.
    \end{align*}
    This allows to pass to the limit and deduce that, \text{for every $t \in (0,\infty)$}, 
    \begin{align*}
       & \frac{1}{2} \int_{\R^2} \rho(t) |u(t)|^2 \dd x - \frac{1}{2} \int_{\R^2} \rho(0) 
        |u(0)|^2 \dd x
       \\ &= \nu \int_0^t \int_{\R^2} |\nabla u(s)|^2 \dd x \dd s
        + \int_{\R^2} \rho(t) |u(t)|^2 \dd x -\int_{\R^2} \rho(0) |u(0)|^2 \dd x, 
    \end{align*}
    which is equivalent to the energy equality.
\end{proof}

\begin{proof}[Proof of \cref{thm:pzz}]
We prove the two implications.

\uline{Part 1:} \emph{\cref{it:strongC} $\implies$ \cref{it:strongIn}.} Since $u \in C([0,\infty); L^2(\R^2))$, if we prove that $u \in L^\infty((0,\infty); H^1(\R^2))$, then it follows that $u \in C([0,\infty); H^1(\R^2))$, which concludes the proof. By the energy equality, there exists a time $T > 0$ for which $u(T) \in H^1(\R^2)$. Due to \cref{rk:matDer}, the momentum equation holds almost everywhere. Convolving it with a spatial mollification kernel, we obtain the equality 
\[
\nu \Delta u_\gamma = (\rho \dot{u})_\gamma + \nabla P_\gamma.
\]
Scalar-multiplying this equation by \( \partial_t u_\gamma \) and integrating in space and time, we get, for every \( t \geq 0 \),
\[
\int_t^T \int_{\R^2} \nu \Delta u_\gamma \cdot \partial_s u_\gamma \, \dd x \, \dd s = \int_t^T \int_{\R^2} (\rho \dot{u})_\gamma \cdot \partial_t u_\gamma + \nabla P_\gamma \cdot \partial_s u_\gamma \, \dd x \, \dd s.
\]
Using integration by parts on the left-hand side and the divergence-free condition yields  
\[
- \frac{\nu}{2} \int_t^T \frac{\dd}{\dd s} \norm{\nabla u(s)}^2_{L^2(\R^2)} \, \dd \tau = \int_t^T \int_{\R^2} (\rho \dot{u})_\gamma \cdot \partial_s u_\gamma \, \dd x \, \dd s.
\]
Since $\partial_t u$ and $\dot{u}$ belong to $L^2((0,\infty)\times \R^2)$, the right-hand side is uniformly bounded in $\gamma$. Hence, for every \( t \geq 0 \),
\[
\norm{\nabla u_\gamma(t)}^2_2 \leq C_1 + \norm{\nabla u_\gamma(T)}^2_2 \leq C_1 + C_2,
\]
where $C_1$ and $C_2$ do not depend on $\gamma$ because \( u(T) \in H^1(\R^2) \). This concludes the proof.

\uline{Part 2:} \emph{\cref{it:strongIn} $\implies$ \cref{it:strongC}}. Owing to~\cite[Theorem 1.1]{PaicuZhangZhang2013} (case $s=1$),  there exists a Leray--Hopf weak solution $(\widetilde{\rho}, \widetilde{u})$ with initial data $(\rho_0, u_0)$ such that $A_1(\widetilde{\rho}, \widetilde{u})$ and $A_2(\widetilde{\rho}, \widetilde{u})$ are bounded. Using \cref{thm:weakstrong1}, we obtain that $(\rho, u) = (\widetilde{\rho}, \widetilde{u})$, as $(\rho,u)$ satisfies the energy equality by \cref{lem:energyeq1}. Then $A_1(\rho, u) = A_1(\widetilde{\rho}, \widetilde{u})$ and $A_2(\rho, u) = A_2(\widetilde{\rho}, \widetilde{u})$, which proves the implication and the additional regularity \cref{def:S}.
\end{proof}

\begin{remark}\label{rk:finrk}
    We can weaken the assumption on $(\rho, u)$ in \cref{thm:pzz}: it is sufficient that $(\rho, u)$ is a Leray--Hopf weak solution that satisfies the energy inequality for almost every \( t > 0 \). Indeed, in \cref{thm:weakstrong1}, it is possible to require only that the weak solution satisfies the energy inequality for almost every \( t > 0 \) (see \cite[Lemma 4.1.]{CrinBaratSkondricViolini2024}). 
\end{remark}

\subsection{Immediately strong solutions} 
\label{ssec:Immstrong}
The definition we adopt for immediately strong solutions is a modification of the definition of strong solutions in \cref{def:strongsol}. In this case, the momentum equation holds in $ L^2((0, \infty);\, L^2(\R^2))$ and not at $t=0$.

\begin{definition}[Immediately strong solutions]\label{def:Immstrongsol}
 We say that a Leray--Hopf solution \( (\rho, u) \) with initial data \( (\rho_0, u_0) \) satisfying \cref{ass:data} is an \textit{immediately strong solution} to \cref{eq:ns} if
\begin{align}\label{eq:ImmstrongCond}
    \partial_t u,\, \nabla^2 u \in L^2((\eps, \infty);\, L^2(\R^2)) \qquad \text{for every $\varepsilon >0$}.
\end{align}
\end{definition}
Unlike the strong solutions that can be produced by $H^1$ initial velocities (see \cref{thm:pzz}), the immediately strong solutions can be derived from $L^2$ initial velocities (see \cref{lem:exist}). As stated in \cref{thm:main2D}, we are interested in the properties satisfied by the immediately strong solutions compared to the ones satisfied by the strong solutions. The counterpart of \cref{def:S} are the following quantities\footnote{~In the presence of vacuum, one can only find bounds for the quantities $\sqrt{\rho} \pt u$ and $\sqrt{\rho} \pt {u}$; however, as we work with densities that are bounded away from $0$, we can drop the prefactor $\sqrt{\rho}.$}:
\begin{align}\label{def:IS}
\begin{aligned}
    & A^0_1(u) \coloneqq  \esssup_{s \in (0, \infty)} s\int_{\R^2} |\nabla u|^2 \dd x + \int_0^\infty s \int_{\R^2} |\partial_s u|^2 + |\nabla^2 u|^2 + \abs{\nabla P}^2 \dd x \dd s,\\
    & A^0_2(u) \coloneqq \esssup_{s \in (0, \infty)} s^2 \int_{\R^2}  |\partial_s u|^2 + |\nabla^2 u|^2 + \abs{\nabla P}^2 \dd x + \int_0^\infty s^2 \int_{\R^2}  |\partial_s \nabla u|^2 \dd x \dd s,\\
    &  A^0_3(u) \coloneqq  \esssup_{s \in (0, \infty)} s^3 \int_{\R^2}  |\nabla \dot{u}|^2 \dd x + \int_0^\infty s^3\int_{\R^2}  |\nabla^2 \dot{u}|^2 + |\nabla \dot{P}|^2 + |\ddot{u}|^2  \dd x \dd s.
    \end{aligned}
\end{align}
The quantity $A^0_3(u)$ was introduced in~\cite[Eq.~(3.33)]{Danchin2024}.
\begin{remark}
It is also possible to consider $u_0 \in H^\eta(\R^2)$, for $\eta>0$, instead of $u_0 \in H^1(\R^2)$. Indeed, in~\cite[Theorem 1.1]{PaicuZhangZhang2013}, the authors prove that starting from $u_0 \in H^\eta(\R^2)$, the estimates \cref{def:S} still hold with the weight $\sigma^{1-\eta}(s)$ in $A_1(u)$ and $\sigma^{2-\eta}(s)$ in $A_2(u)$. In this sense, $A^0_1(u)$ and $A^0_2(u)$ correspond to the limit case $\eta=1$.   
 \end{remark}

\section{Equivalence among \texorpdfstring{\cref{it:immediately}, \cref{it:strong}, and \cref{it:A1A2A3}}{(i), (ii), and (iii)}}
\label{sec:i-ii-iii}

\subsection{Strong energy inequality}\label{subsec:SEI}

In this section, we show that a Leray--Hopf weak solution is immediately strong if and only if it satisfies the strong energy inequality, \ie, \cref{it:immediately} $\iff$ \cref{it:strong}. This is strongly related to the semi-flow property of Leray--Hopf weak solution (which holds if they satisfy the strong energy inequality): if we consider a time translation of the solution, it is still a Leray--Hopf weak solution with translated initial data. 

We start with some remarks on the set of continuity points of the momentum.

\begin{lemma}\label{lem:semi-flow}
    Let $(\rho,u)$ be a Leray--Hopf solution of \cref{eq:ns}. There exists a measurable set $I \subset [0,\infty)$ such that $0 \in I$, $\mathscr{L}^1([0,\infty) \setminus I)=0$ and \footnote{~By $ \rho u \in C_w(I;\,L^2_\sigma)$, we mean that, for every sequence $\{ t_n \}_{n\in \mathbb N} \subset I$ such that $ t_n \to t \in I$ and for every $\varphi \in L_\sigma^2(\mathbb{R}^2)$, we have
    \begin{align*}
        \int_{\mathbb{R}^2} \rho (t_n,x) u (t_n,x) \cdot \varphi (x) \, \mathrm{d}x \rightarrow \int_{\mathbb{R}^2} \rho(t,x) u(t,x) \cdot \varphi (x) \, \mathrm{d}x.
    \end{align*}}
    \begin{align*}
        \rho u \in C_w(I;\,L^2_\sigma).    
    \end{align*}
    Furthermore, we can modify $\rho$ on a negligible set of times such that
    \begin{align*}
        \rho \in C_{w^*}([0,\infty);\,L^\infty(\R^2)).
    \end{align*}
\end{lemma}

\begin{proof}
    For the first part, we refer to~\cite[Lemma 3.2]{CrinBaratSkondricViolini2024} and for the second part, we note that $\rho$ is a solution of the transport equation and apply~\cite[Remark 2.2.2]{Crippa2008}.
\end{proof}
\begin{corollary}\label{cor:distSol}
    Let $(\rho,u)$ be a Leray--Hopf solution of \cref{eq:ns} and let $t_0 \in I$, with $I$ is as in \cref{lem:semi-flow}. Let us define 
    $(\rho_{t_0},u_{t_0})$ by
    \begin{align*}
        \rho_{t_0}(t)\coloneqq \rho(t+t_0),\qquad u_{t_0}\coloneqq u(t+t_0).
    \end{align*}
    Then $(\rho_{t_0},u_{t_0})$ is a distributional solution of \cref{eq:ns} with initial data $(\rho (t_0), u(t_0)) $.
\end{corollary}
\begin{proof}
Let $t_0 \in I$ with $t_0>0.$ We have to show that, for every $\varphi \in C^\infty_{c,\sigma}([0,\infty)\times \R^2 )$, 
\begin{align*}
   - \int_{\R^2}&  \rho(t_0) u(t_0) \cdot \varphi(t_0) \, \dd x  \\
    &= \int_{t_0}^\infty \int_{\R^2} \rho \, \partial_s \varphi \cdot u +
    \rho u \otimes u : \nabla \varphi \, \dd x \, \mathrm{d} s
    - \nu \int_{t_0}^\infty \int_{\R^2} \nabla u : \nabla \varphi \, \dd x \, \dd s.
\end{align*}
To this end, we let $\varphi \in C^\infty_{c,\sigma}([0,\infty)\times \R^2)$ and define
\begin{align*}
    \widetilde{\varphi}(t)\coloneqq
    \begin{cases}
        \varphi(t_0) , & \text{if } t<t_0,\\
        \varphi(t) , & \text{if } t_0 \leq t.
    \end{cases}
\end{align*}
Then $\widetilde{\varphi}$ is an admissible test function for the weak formulation of the momentum equation for $(\rho,u)$ (see \cite[Lemma 3.3]{CrinBaratSkondricViolini2024}), and we have 
\begin{align*}
    \partial_t \widetilde{\varphi} = \mathds{1}_{[t_0,\infty)} \partial_t \varphi.
\end{align*}
We obtain
\begin{align*}
    - \int_{0}^{\infty} \int_{\R^2} \rho \partial_s \widetilde{\varphi} \cdot u + \rho u \otimes u : \nabla \widetilde{\varphi} \dd x \dd s
    = - \nu \int_{0}^\infty \int_{\R^2} \nabla u : \nabla \widetilde{\varphi} \dd x \dd s
    + \int_{\R^2} \rho_0 u_0 \cdot \widetilde{\varphi}(0) \dd x,
\end{align*}
which is equivalent to
\begin{align}\label{eq:weakformvarphi1}
\begin{aligned}
    \int_{\R^2} \rho_0 u_0 \cdot \varphi(t_0) \dd x =
    - & \int_{t_0}^{\infty} \int_{\R^2} \rho \partial_s \varphi \cdot u + \rho u \otimes u : \nabla \varphi \dd x \dd s + \nu \int_{t_0}^\infty \int_{\R^2} \nabla u : \nabla \widetilde{\varphi} \dd x \dd s\\
    & = - \nu \int_{0}^{t_0} \int_{\R^2} \nabla u : \nabla \varphi(t_0) \dd x \dd s
    + \int_{0}^{t_0} \int_{\R^2}  \rho u \otimes u : \nabla \varphi(t_0) \dd x \dd s.
\end{aligned}
\end{align}
Since $0,t_0 \in I$, we know, from \cite[Lemma 3.2]{CrinBaratSkondricViolini2024}, that
\begin{align}\label{eq:weakformvarphi2}
\begin{aligned}
    \int_{\R^2} \rho(t_0) u(t_0) \cdot \varphi(t_0) \dd x =&
    - \nu \int_{0}^{t_0} \int_{\R^2} \nabla u : \nabla \varphi(t_0) \dd x \dd s\\
    &+ \int_{0}^{t_0} \int_{\R^2}  \rho u \otimes u : \nabla \varphi(t_0) \dd x \dd s
    + \int_{\R^2} \rho_0 u_0 \cdot \varphi(t_0) \dd x.
\end{aligned}
\end{align}
Inserting \cref{eq:weakformvarphi2} into \cref{eq:weakformvarphi1} completes the proof.
\end{proof}

We are now ready to prove that \cref{it:strong} $\implies$ \cref{it:immediately}. A priori, $(\rho_{t_0},u_{t_0})$ needs not be a Leray--Hopf weak solution (indeed, it is unclear whether the strong energy inequality holds in this case). However, if we assume that $(\rho, u)$ satisfies the strong energy inequality, then  $(\rho_{t_0},u_{t_0})$ is a Leray--Hopf weak solution, which, in turn, yields the semi-flow property. This semi-flow property, from a different perspective, can be interpreted as a regularization of $(\rho, u)$ immediately after time $0$ (see the proof of \cref{prop:contitimeSEI}). This regularization is the rationale behind the definition of an immediately strong solution.

\begin{proposition}[Implication \cref{it:strong} $\implies$ \cref{it:immediately}] \label{prop:contitimeSEI}
    Let $(\rho,u)$ be a Leray--Hopf solution of \cref{eq:ns} and assume that $(\rho,u)$ satisfies the strong energy inequality. Then
    \begin{enumerate}
        \item $(\rho,u)$ is an immediately strong solution;
        \item  we can modify $u$ on a negligible set of times such that $u \in C((0,\infty);\,L^2(\R^2));$
        \item after the modification above, for every  $\eps > 0$, the translation 
        \begin{align*}
         \rho_\eps(t,x) \coloneqq \rho(t+\eps,x),\quad u_\eps(t,x)\coloneqq u(t+\eps,x)
        \end{align*}
        is a strong solution of \cref{eq:ns} with initial data $(\rho(\eps,x),u(\eps,x))$.
    \end{enumerate}
\end{proposition}

\begin{proof}
   Let $J$ be the set of time points for which the strong energy inequality holds. Denote by $Z$ the set of points $s$ for which $u(s) \in H^1(\mathbb{R}^2)$, and by $I$ the set of points introduced in \cref{lem:semi-flow}. Let
    $I_w \coloneqq I \cap J \cap Z$.

   \textbf{Step 1.} We claim that, for every $t_0 \in I_w$, the time translation
    \begin{align*}
        \rho_{t_0}(t,x)\coloneqq \rho(t+t_0,x), \quad u_{t_0}(t,x)\coloneqq u(t+t_0,x)
    \end{align*}
    is a strong solution of \cref{eq:ns} in $[0,\infty)$ with initial data $(\rho(t_0),u(t_0))$. First, we prove that $(\rho_{t_0},u_{t_0})$ is a Leray--Hopf weak solution in $[t_0,\infty) \times \R^2.$ By \cref{cor:distSol}, $(\rho_{t_0}, u_{t_0})$ is a distributional solution, and since $t_0 \in J$, the pair $(\rho_{t_0}, u_{t_0})$ satisfies the energy inequality, \ie, for every $t > 0$ such that $t+t_0 \in J$, we have 
    \begin{align*}
        \frac{1}{2} \int_{\R^2} \rho_{t_0}(t) |u_{t_0}(t)|^2 \dd x + \nu \int_0^t \int_{\R^2} |\nabla u_{t_0}(s)|^2 \dd x \dd s \leq \frac{1}{2} \int_{\R^2} \rho(t_0) |u(t_0)|^2 \dd x.
    \end{align*}
    Hence, $(\rho_{t_0},u_{t_0})$ is a Leray--Hopf weak solution in $[t_0,\infty) \times \R^2.$ Since $u(t_0) \in H^1(\R^2)$, we deduce that $(\rho_{t_0},u_{t_0})$ is also a strong solution in $[t_0,\infty) \times \R^2$ by \cref{thm:pzz}.

\textbf{Step 2.} We claim that we can modify $u$ on a negligible set of times such that $u \in C((0,\infty);L^2(\R^2))$ and $(0,\infty) \subset I$. To do so, notice that, as discussed before, for every $\eps > 0$, we can modify $u$ to get $u \in C([\eps,\infty);\, L^2_\sigma(\R^2))$. This follows from the fact that $I_w$ is dense in $[0,\infty)$ and, for every $\eps>0$, we can find a $t_\eps \in (0,\eps)$ and consider a strong solution which is continuous in time and coincides with $(\rho,u)$ almost everywhere in $[t_\eps,\infty) \times \R^2.$
     
To prove that the modification does not depend on the choice of $\eps > 0$, let us consider $0 < \eps_1 < \eps_2$ and the respective modifications $u^1$ and $u^2$. Since $u^i \equiv u$ almost everywhere in $(0,\eps_i) \times \R^2$ and since $u_{\eps_i} \in C([\eps_i,\infty);\, L^2(\R^2))$ for $i \in \{1,2\}$, we conclude that
\begin{align*}
        u_{\eps_1} \equiv u_{\eps_2} \quad \text{in } [\eps_2,\infty) \times \R^2.
\end{align*}
Moreover, since $\rho \in C_{w^*}((0,\infty);L^\infty(\R^2))$, we get $\rho u \in C_w((0,\infty);L_\sigma^2(\R^2))$ and so $(0,\infty) \subset I$.

\textbf{Step 3.} We prove that $(0,\infty) \subset I_w$. Since $\mathscr{L}^1([0,\infty) \setminus I_w) = 0$, for every $\eps > 0$, there exists $t_\eps \in I$ such that $0 < t_\eps < \eps$. Since $(\rho_{t_\eps}, u_{t_\eps})$ is a strong solution, using \cref{lem:energyeq1}, we obtain the energy equality in $[0,\infty)$ for $(\rho_{t_\eps}, u_{t_\eps})$, which implies the energy equality for $(\rho, u)$ in $[t_\eps,\infty)$, hence $\eps \in J$. By \cref{thm:pzz} $(\rho_{t_\eps}, u_{t_\eps}) \in C([0,\infty);H^1(\R^2))$ which leads to $(\rho, u) \in C([t_\eps,\infty);H^1(\R^2))$ and thus $\eps \in Z$.
    
\end{proof}

The proof of the reverse implication, \cref{it:immediately} $\implies$ \cref{it:strong}, follows the line of \cref{lem:energyeq1} and~\cite[Lemma 3.3]{CrinBaratSkondricViolini2024}.

\begin{proposition}[Implication \cref{it:immediately} $\implies$ \cref{it:strong}] \label{prop:energySIn}
    Let $(\rho,u)$ be an immediately strong solution of \cref{eq:ns}. Then $(\rho,u)$ satisfies the strong energy inequality.
\end{proposition}

\begin{proof}
  Let us consider $s>0$ and pick $\eps$ such that $0<\eps<s$. By definition, we have
    \begin{align*}
        \pt u,\, \nabla^2 u \in L^2((\eps,\infty);L^2(\R^2)),
    \end{align*}
    and, similarly to \cref{lem:energyeq1}, we get that $(\rho,u)$ satisfies the energy equality in $(\eps,\infty)$. 

\end{proof}
\begin{remark}\label{rk:AdmCont}
    We have seen that if $(\rho,u)$ is an immediately strong solution then $u \in C((0,\infty);L^2(\R^2))$ and this continuity plays a crucial role in the energy estimates. Indeed, following the proof of~\cite[Lemma 3.3]{CrinBaratSkondricViolini2024}, it guarantees us that the solution is an admissible test function in the weak formulation of both transport and momentum equations in \cref{eq:ns} for positive times. Unfortunately, the continuity at the initial time for the velocity is unknown for immediately strong solutions and this prevents us from proving the energy equality. This continuity at time $t=0$ is also a crucial assumption to prove uniqueness in the class of immediately strong solutions as required in \cref{thm:weakstrong2D}.
\end{remark}

\subsection{Time-decay estimates}\label{subse:Timedec}

In this subsection, we show that \cref{it:A1A2A3} is also equivalent to \cref{it:immediately} and \cref{it:strong}. First, we note that the implication \cref{it:A1A2A3} $\implies$ \cref{it:strong} is directly encoded in the definition of $A_1$. The reverse implication is proven in the following proposition.

\begin{proposition}[Implication \cref{it:strong} $\implies$ \cref{it:A1A2A3}]
Let $(\rho,u)$ be a Leray--Hopf solution of \cref{eq:ns} tsatisfying the strong energy inequality (in the sense of \cref{def:energies}). There exists a constant 
    $C=C(\nu,\norm{u_0}_{L^2},c_0,C_0)$ such that
    \begin{align*}
        A_1^0(u), \, A_2^0(u), \, A_3^0(u) \leq C,
    \end{align*}
    where $A^0_i$, for $i \in \{1,2,3\}$, is defined in \cref{def:IS}.
\end{proposition}

\begin{proof}
Thanks to the results in \cref{subsec:SEI}, we already know that \cref{it:strong} $\iff$ \cref{it:immediately}. By \cref{prop:contitimeSEI}, for every $\varepsilon > 0$, the time-shifted functions
\[
  u_\varepsilon(t,x) \coloneqq u(t+\varepsilon,x) \quad \text{and} \quad \rho_\varepsilon(t,x) \coloneqq \rho(t+\varepsilon,x)
\]
are strong solutions with initial data $(\rho_\varepsilon(0), u_\varepsilon(0))$. Furthermore, by \cref{thm:pzz}, we have $u_\varepsilon(0) \in H^1(\mathbb{R}^2)$. Following the proof of~\cite[Theorem 1.1]{Danchin2024}, we obtain the existence of a Leray--Hopf weak solution $(\widetilde{\rho}_\varepsilon, \widetilde{u}_\varepsilon)$ with initial data $(\rho_\varepsilon(0), u_\varepsilon(0))$ such that
\begin{align}\label{eq:Init}
    A_i^0(\widetilde{u}_\varepsilon) \leq C \quad \text{for } i \in \{1, 2, 3\},
\end{align}
where $C$ depends only on the constants $c_0$ and $C^0$ defined in \cref{ass:data} and $\norm{u_\varepsilon(0)}_{L^2}$%
\footnote{~In the case of smooth triplet $(\rho,u,P)$ the bounds for $A_1^0$, $A_2^0$, and $A_3^0$ are proved in~\cite[(2.11), (2.21), (2.26)]{Danchin2024}. In~\cite[Section 2.2]{Danchin2024}, it is shown that the bound depends only on $c_0$, $C^0$, and the $L^2$ norm of the initial data. The non-smooth case is also treated in~\cite[Section 2.2]{Danchin2024} using an approximation argument.}.
Moreover, $C$ is an increasing function of $\norm{u_\varepsilon(0)}_{L^2}$. By weak-strong uniqueness (see \cref{thm:weakstrong1}), we conclude that $(\widetilde{\rho}_\varepsilon, \widetilde{u}_\varepsilon) \equiv (\rho_\varepsilon, u_\varepsilon)$, and thus \cref{eq:Init} holds with $\widetilde{u}_\varepsilon$ replaced by $u_\varepsilon$. 

To conclude the proof, we shift back in time. We focus only on the term $A_1^0$, as the other two can be handled similarly, and compute 
\begin{align*}
    A^0_1(u_\varepsilon) &= \esssup_{s \in [0, T]} s \int_{\mathbb{R}^2} |\nabla u_\varepsilon (s)|^2 \, \mathrm{d}x + \int_0^T s \int_{\mathbb{R}^2} \left( |\partial_s u_\varepsilon|^2 + |\nabla^2 u_\varepsilon|^2 + |\nabla P_\varepsilon|^2 \right) \, \mathrm{d}x \, \mathrm{d}s, \\ 
    &= \esssup_{s \in [\varepsilon, T+\varepsilon]} s\int_{\mathbb{R}^2} |\nabla u (s)|^2 \, \mathrm{d}x + \int_\varepsilon^{T+\varepsilon} s \int_{\mathbb{R}^2} \left( |\partial_s u|^2 + |\nabla^2 u|^2 + |\nabla P|^2 \right) \, \mathrm{d}x \, \mathrm{d}s.
\end{align*}
Since $\norm{u_\varepsilon(0)}_{L^2} \leq C \norm{u(0)}_{L^2}$ by the energy inequality and the no-vacuum assumption, the right-hand side of \cref{eq:Init} does not depend on the time shift $\varepsilon$. Thus, we can take the limit as $\varepsilon \to 0$ in $A^0_1(u_\varepsilon)$ and the bound is preserved. Since $A^0_1(u_\varepsilon) \to A^0_1(u)$ as $\varepsilon\to0$, the proof is complete by the freedom of $T>0$.

\end{proof}

\section{Role of the pressure and proof of \texorpdfstring{\cref{it:immediately}, \cref{it:strong}, \cref{it:A1A2A3}$\iff$\cref{it:pressure}}{(i), (ii), (iii) if and only if (iv)}}
\label{sec:pressure}
\subsection{Associated pressure and suitable solutions}

We begin by specifying that when referring to the pressure \emph{associated} with a Leray--Hopf solution $(\rho, u)$ of \cref{eq:ns}, we mean the existence of a pressure $P$ such that the momentum equation is satisfied in the sense of distributions.

\begin{definition}[Associated pressure]\label{def:assPress}
    Let \( (\rho,u) \) be a Leray--Hopf solution. We say that \( P \in 
    L^1_\loc((0,\infty) \times \R^2) \) is a \emph{pressure associated with \( (\rho,u) \)} if, for every \( \varphi \in C_{c}^\infty((0,\infty) \times \R^2;\R^2) \),
    \begin{align*}
        - \int_0^\infty \int_{\R^2} &\rho \partial_t \varphi \cdot u +
        \rho u \otimes u : \nabla \varphi \dd x \dd s  
         + \nu \int_0^\infty \int_{\R^2} \nabla u : \nabla \varphi \dd x \dd s = 
        \int_0^\infty \int_{\R^2} P \dive \varphi \dd x \dd s.
    \end{align*}
\end{definition}

We also recall the notion of \emph{suitable solution}. The suitability condition can be interpreted as local energy conservation, which is obtained by taking the scalar product of \cref{eq:ns} with \( u \).

\begin{definition}[Suitable solutions]\label{def:suitable}
We say that \( (\rho,u,P) \) is \emph{suitable} if \( (\rho,u) \) is a {Leray--Hopf weak solution} and \( P \) is an associated {pressure} such that \( P u \in L^1_\loc((0,\infty) \times \R^2) \) and the following equality holds in the sense of distributions:
\begin{align*}
    \partial_t \left(\frac{\rho \abs{u}^2}{2}\right) - \nu \Delta\left( \frac{\abs{u}^2}{2}\right) + \nu \abs{\nabla u}^2  + &  \div \left(\frac{1}{2} \rho \abs{u}^2 u + P u \right) = 0,
\end{align*}
meaning that for every \( \varphi \in C^\infty_c((0,\infty) \times \R^2) \), we have 
\begin{align*}
    \begin{aligned}
        - \int_0^\infty\int_{\R^2} \frac{1}{2} \rho |u|^2 \partial_t \varphi  - \nu & 
        \int_0^\infty\int_{\R^2} \frac{1}{2}|u|^2 \Delta \varphi + \nu 
        \int_0^\infty\int_{\R^2} |\nabla u|^2 \varphi  - \int_0^\infty\int_{\R^2} \left(\frac{1}{2} \rho |u|^2  u + P  u \right) \cdot \nabla \varphi = 0.
    \end{aligned}
\end{align*}
\end{definition}

The definition of suitability can be extended to a sequence of test functions, provided we restrict to a time interval of full measure.
\begin{lemma}\label{lem:suit}
    Suppose \( (\rho,u,P) \) is suitable and consider a sequence \( \{\varphi_n\}_{n \in\mathbb N} \subset C^\infty_c(\R^2) \). There exists a set \( I \subset (0,\infty) \), with \( \mathscr{L}^1((0,\infty) \setminus I) = 0 \), such that, for every \( n \in \N \) and every \( t,s \in I \) with \( s<t \),
    \begin{align}\label{eq:timeindependent}
    \begin{aligned}
        \int_{\R^2} \frac{1}{2} \rho(t) |u(t)|^2 \varphi_n \dd x & - \int_{\R^2} \frac{1}{2} \rho(s) |u(s)|^2 \varphi_n \dd x =   \nu 
        \int_s^t \int_{\R^2} \frac{1}{2}|u|^2 \Delta \varphi_n \dd x \dd \tau  \\
       & - \nu
        \int_s^t \int_{\R^2} |\nabla u|^2 \varphi_n \dd x \dd \tau + \int_s^t \int_{\R^2} \left(\frac{1}{2} \rho |u|^2  u + P u \right) \cdot \nabla \varphi_n \dd x 
        \dd \tau. 
    \end{aligned}
    \end{align}
\end{lemma}
\begin{proof}
    For every \( n \in \N \), let \( I_n \) be the set of Lebesgue points of \( t \mapsto \frac{1}{2} \int_{\R^2} \rho(t) |u(t)|^2 \varphi_n \dd x \), and set $I \coloneqq \bigcap_{n \in \N} I_n$. As in the proof of~\cite[Lemma 3.2]{CrinBaratSkondricViolini2024}, we can use Lebesgue's  differentiation theorem to show that \cref{eq:timeindependent} is true for every \( t,s \in I \) and every 
    \( n \in \N \).
\end{proof}

\subsection{Pressure regularity and properties of the energy}
\label{ssec:pr1}

In this subsection, we show that the regularity of the pressure \cref{it:pressure} implies the strong energy inequality \cref{it:strong}. We divide this proof into two steps. First, in \cref{lem:suitable}, we show that, if there exists an associated pressure \( P \in L^2_\loc((0, \infty) \times \R^2) \), then the triplet \( (\rho,u,P) \) is a suitable solution (in the sense of \cref{def:suitable}), meaning that it conserves the energy locally. For the proof of this fact, we follow the strategy used in~\cite{BieWang2023}. Secondly, in \cref{lem:StrongIneq}, we prove that, under the additional assumption \( P \in L^2_\loc((0, \infty); \bmo(\R^2)) \), the solution \( (\rho,u) \) satisfies the strong energy inequality.

\begin{lemma}\label{lem:suitable}
    Let $(\rho, u)$ be a Leray--Hopf solution and assume that there exists an associated pressure $P$ in the class $ L_\loc^2((0,\infty)\times \R^2)$. Then the triplet $(\rho, u, P)$ is suitable. 
\end{lemma}

\begin{proof}
    We divide the proof into five steps.

    \textbf{Step 1.} \emph{Regularization via spatial mollification.} Convolving in space the terms in \cref{eq:ns}, we deduce that the following identities hold almost everywhere:
    \begin{align*}
        \partial_t \rho_\gamma + \dive((\rho u)_\gamma) = 0 \quad \text{and}\quad 
        \partial_t (\rho u)_\gamma + \dive( (\rho u \otimes u)_\gamma) + \nabla P_\gamma = \nu \Delta u_\gamma.
    \end{align*}
    Let \(\varphi \in C^\infty_c((0,+\infty) \times \R^2)\) and fix compact sets \(I \subset (0,\infty)\) and \(K \subset \R^2\) such that \(\supp  \varphi  \subset I \times K\). Multiplying the regularized momentum equation by 
    \(
    \frac{(\rho u)_\gamma}{\rho_\gamma}  \varphi,
    \)
    and integrating over space-time, we obtain
    \begin{align}\label{eq:ccss}
    \begin{aligned}
        &\iint_{I \times K} \partial_s (\rho u)_\gamma \cdot \frac{(\rho u)_\gamma}{\rho_\gamma} \varphi \, \mathrm d x \, \mathrm d s   \\ &= \underbrace{- \iint_{I \times K}   \left( \dive((\rho u \otimes u)_\gamma) \cdot \frac{(\rho u)_\gamma}{\rho_\gamma} \varphi -   \nabla P_\gamma \cdot \frac{(\rho u)_\gamma}{\rho_\gamma} \varphi   + \nu  \Delta u_\gamma \cdot \frac{(\rho u)_\gamma}{\rho_\gamma} \varphi \right) \, \mathrm d x \, \mathrm d s 
 }_{\eqqcolon R_\gamma}.
    \end{aligned}
    \end{align}
    Our aim is to let $\gamma \to 0$ in \cref{eq:ccss}. 

    \textbf{Step 2.} \emph{The right-hand side.} Using integration by parts on the right-hand side of \cref{eq:ccss}, we get
    \begin{align*}
        R_\gamma &= \iint_{I \times K} (\rho u \otimes u)_\gamma : \nabla \left( \frac{(\rho u)_\gamma}{\rho_\gamma} \varphi \right)  \, \mathrm d x \, \mathrm d s + \iint_{I \times K}  P_\gamma \div  \left( \frac{(\rho u)_\gamma}{\rho_\gamma} \varphi \right)  \, \mathrm d x \, \mathrm d s 
        \\ & \quad - \nu \iint_{I \times K} \nabla u_\gamma : \nabla  \left( \frac{(\rho u)_\gamma}{\rho_\gamma} \varphi \right) \, \mathrm d x \, \mathrm d s.
    \end{align*}
We claim that, as $\gamma\to0$,
    \begin{align}\label{eq:claim}
        \norm{\nabla \left( \frac{(\rho u)_\gamma}{\rho_\gamma} \right) - \nabla u \ }_{L^2(I \times K)} \longrightarrow 0.
    \end{align}
We will prove \eqref{eq:claim} in Step 5. It is straightforward to check that, using \cref{eq:claim}, we have that
    \begin{align}\label{eq:hardconv}
        \norm{\nabla \left( \frac{(\rho u)_\gamma}{\rho_\gamma} \varphi \right) - \nabla( u \varphi)}_{L^2} \longrightarrow 0 \quad \text{and} \quad  \norm{\div \left( \frac{(\rho u)_\gamma}{\rho_\gamma} \varphi \right) - \div ( u \varphi)}_{L^2} \longrightarrow 0.
    \end{align}
    The pressure belongs to \(L^2_\loc ((0, \infty) \times \R^2)\) by assumption and \(\nabla u \in L^2_\loc ((0, \infty) \times \R^2)\) by the energy inequality. Then \(u \in L_\loc^4((0,\infty) \times \R^2)\) due to  Ladyzhenskaya's inequality, and thus
    \begin{align*}
       \norm{(\rho u \otimes u)}_{L^2}  
        &\leq \norm{\rho}_{L^\infty} \norm{u}^2_{L^4} \leq C.
    \end{align*}
    By the properties of the mollification, we have that 
    \begin{align}\label{eq:easycon}
       (\rho u \otimes u)_\gamma \longrightarrow (\rho u \otimes u), \quad P_\gamma \longrightarrow P, \quad \text{and} \quad \nabla u_\gamma \longrightarrow \nabla u \qquad \text{in \(L^2(I \times K)\)}.
    \end{align}
    Combining \cref{eq:hardconv} and \cref{eq:easycon}, we get that
    \begin{equation}\label{eq:Rths}
      \begin{aligned}
        \lim_{\gamma \to 0}R_\gamma & = \iint_{I \times K} \left( \rho u \otimes u : \nabla \left( u \varphi \right) +   P \div  \left( u \varphi \right)  -\nu  \nabla u : \nabla  \left( u \varphi \right) \right) \, \mathrm d x \, \mathrm d s  \\  
        & =  \iint_{I \times K} \left(  \varphi \rho u \otimes u : \nabla u +  \rho \abs{u}^2 u \cdot \nabla \varphi +   P u \cdot \nabla \varphi  - \nu  \abs{\nabla u}^2 \varphi + \nu  \frac{\abs{u}^2}{2} \Delta \varphi \right) \, \mathrm d x \, \mathrm d s,
      \end{aligned}
    \end{equation}
    where we used the fact that
    \begin{align*}
        - \iint_{I \times K} \nabla u : \nabla  \varphi \otimes  u \, \mathrm d x \, \mathrm d s =  \iint_{I \times K} \frac{\abs{u}^2}{2} \Delta \varphi \, \mathrm d x \, \mathrm d s.
    \end{align*}

    \textbf{Step 3.} \emph{Time derivative.} Integrating by parts in time, we get 
    \begin{align*} 
    \iint_{I \times K} \partial_s \left(\frac{1}{2}\left|(\rho u)_{\gamma}\right|^2\right) \frac{\varphi}{\rho_{\gamma}} \, \mathrm d x \, \mathrm d s= \underbrace{- \frac{1}{2}\iint_{I \times K} \left|(\rho u)_{\gamma}\right|^2 \partial_t \varphi \frac{1}{\rho_{\gamma}} \, \mathrm d x \, \mathrm d s}_{\eqqcolon T_{1,\gamma}} +\underbrace{\frac{1}{2}\iint_{I \times K} \left|(\rho u)_{\gamma}\right|^2 \varphi \frac{\partial_s \rho_{\gamma}}{\rho_{\gamma}^2} \, \mathrm d x \, \mathrm d s}_{\eqqcolon T_{2,\gamma}}.
    \end{align*}
    The convergence of \(T_{1,\gamma}\) is clear: as $\gamma \to0$, we have
    \begin{align}\label{eq:T1ths}
        T_{1,\gamma} \longrightarrow -\frac{1}{2} \iint_{I \times K} \rho |u|^2 \partial_s \varphi \, \mathrm d x \, \mathrm d s.
    \end{align}
    For \(T_{2,\gamma}\), we first use the transport equation \(\partial_t \rho_\gamma = - \div (\rho u)_\gamma\)  and integration by parts to write
    \begin{equation}\label{eq:T2ths}
      \begin{aligned}
        \lim_{\gamma \to 0} T_{2,\gamma} & = -\frac{1}{2} \lim_{\gamma \to 0} \left(\iint_{I \times K} \left|(\rho u)_{\gamma}\right|^2 \varphi \frac{\div ((\rho u)_\gamma)}{\rho_\gamma^2}  \, \mathrm d x \, \mathrm d s\right) \\ 
        & = \frac{1}{2} \lim_{\gamma \to 0} \left(\iint_{I \times K} \nabla \left( \left|(\rho u)_{\gamma}\right|^2 \frac{\varphi}{\rho_{\gamma}^2}  \right) \cdot (\rho u)_\gamma \, \mathrm d x \, \mathrm d s \right) \\ 
        & = \frac{1}{2}  \lim_{\gamma \to 0} \left( \iint_{I \times K} \frac{\left|(\rho u)_{\gamma}\right|^2 }{\rho_{\gamma}}  \nabla \varphi \cdot   u \, \mathrm d x \, \mathrm d s \right) + \frac{1}{2}\lim_{\gamma \to 0} \left( \iint_{I \times K} \varphi   \nabla \left( \frac{\left|(\rho u)_{\gamma}\right|^2}{\rho_\gamma^2} \right) \cdot (\rho u)_\gamma \, \mathrm d x \, \mathrm d s \right)  \\ 
        & = \frac{1}{2}  \iint_{I \times K} \rho \abs{u}^2 \nabla \varphi \cdot  u  \, \mathrm d x \, \mathrm d s + \lim_{\gamma \to 0} \left( \iint_{I \times K} \varphi  (\rho u)_\gamma \otimes \frac{(\rho u)_\gamma}{\rho_\gamma} :  \nabla \left( \frac{(\rho u)_{\gamma}}{\rho_\gamma} \right) \, \mathrm d x \, \mathrm d s \right) \\ 
        & = \frac{1}{2}  \iint_{I \times K} \rho \abs{u}^2 \nabla \varphi \cdot  u \, \mathrm d x \, \mathrm d s +  \iint_{I \times K} \varphi  \rho u \otimes u  :  \nabla  u \, \mathrm d x \, \mathrm d s.
       \end{aligned}
    \end{equation}
    To prove the last limit in \cref{eq:T2ths}, we used the fact that
    \begin{align*}
        \norm{(\rho u)_\gamma -\rho u}_{L^4} \longrightarrow 0, \qquad \norm{\frac{(\rho u)_\gamma}{\rho_\gamma} - u}_{L^4} \longrightarrow 0 \quad \text{and} \quad \norm{\nabla \left( \frac{(\rho u)_{\gamma}}{\rho_\gamma}\right) - \nabla u}_{L^2} \longrightarrow 0,
    \end{align*}
    where the first one follows from $\rho u \in L^4(I \times K)$, the last one by \cref{eq:claim}, and, for the second one, we use again Ladyzhenskaya's inequality to obtain
    \begin{align*}
        \norm{\frac{(\rho u)_\gamma}{\rho_\gamma} - u}_{L^4}& \leq \norm{\frac{(\rho u)_\gamma}{\rho_\gamma} - u}^{1\slash 2}_{L^2} \ \norm{\nabla \left(\frac{(\rho u)_\gamma}{\rho_\gamma} \right)- \nabla u}^{1\slash 2}_{L^2} 
    \end{align*}
    which converges to zero again by \cref{eq:claim}.

    \textbf{Step 4.} \emph{Conclusion.} We let $\gamma \to 0$ in  \cref{eq:ccss}: using \cref{eq:Rths}, \cref{eq:T1ths}, and \cref{eq:T2ths}, we conclude that 
    \begin{align*}
        \partial_t \left(\frac{\rho \abs{u}^2}{2}\right) - \nu \Delta \left(\frac{\abs{u}^2}{2} \right) +  \nu \abs{\nabla u}^2  + \div \left(\frac{1}{2}\rho \abs{u}^2 u + P u \right) = 0
    \end{align*}
    holds in the sense of distributions,  which is the definition of suitability according to \cref{def:suitable}.

    \textbf{Step 5.} \emph{Proof of \cref{eq:claim}.} We apply the chain rule and use the fact that \(\rho_\gamma\) is bounded from below by a positive constant:
    \begin{align*}
        \norm{\nabla \left( \frac{(\rho u)_\gamma}{\rho_\gamma} \right) - \nabla u }_{L^2} 
        & =   \norm{\nabla \left( \frac{(\rho u)_\gamma -\rho_\gamma u}{\rho_\gamma} \right)}_{L^2} \\
        & \leq \norm{\frac{1}{\rho_\gamma}}_{L^\infty} \norm{\nabla \left( (\rho u)_\gamma -\rho_\gamma u \right)}_{L^2} + \norm{(\rho u)_\gamma -\rho_\gamma u}_{L^2} \norm{\frac{1}{\rho_\gamma^2}}_{L^\infty} \norm{\nabla \rho_\gamma}_{L^\infty} \\
        & \lesssim \norm{\nabla \left( (\rho u)_\gamma -\rho_\gamma u \right)}_{L^2} + \norm{\frac{(\rho u)_\gamma -\rho_\gamma u}{\gamma}}_{L^2} \norm{\gamma \nabla \rho_\gamma}_{L^\infty},
    \end{align*}
    which converges to zero by \crefrange{lem:commutator1}{lem:deriv-conv}.
\end{proof}

\begin{lemma}\label{lem:StrongIneq}
      Let $(\rho, u)$ be a Leray--Hopf solution of \cref{eq:ns}, assume that $P \in L_\loc^2((0,\infty);\,\bmo (\R^2))$ and that $(\rho,u,P)$ is suitable. Then $(\rho, u)$ satisfies the strong energy inequality.
\end{lemma}
\begin{proof}
We will prove that for every $s, t \in I$, with $I$ defined in \cref{lem:suit}, the energy equality holds. 
Let $\varphi \in C_c^\infty (\R^2)$ such that $0\leq \varphi \leq 1$, $\supp \varphi \subset B(0,1)$ and $\varphi \equiv 1$ in $B(0,1\slash 2)$.  
Then the sequence $\varphi_n(x)\coloneqq \varphi(x\slash n)$ inherits the following properties:
\begin{enumerate}
    \item For every $x \in \R^2$, we have $\varphi_n (x) \to  1$ as $n \to \infty$;
    \item  $\nabla \varphi_n$ and $\Delta \varphi_n$ are supported in the annulus $A_n \coloneqq \left\{ x \in B(0,n)  : \; n\slash 2 < \abs{x}<  n   \right\}$ and
    \begin{align}\label{eq:GradLap}
        |\nabla \varphi_n| \lesssim \frac{1}{n}, \qquad |\Delta \varphi_n| \lesssim \frac{1}{n^2}.
    \end{align}
\end{enumerate}
Using \cref{lem:suit} with $\left\{ \varphi_n \right\}_{n\in \N}$, we get that, for every $s,t \in I$ (with $I$ defined in \cref{lem:suit}), 
\begin{equation}\label{eq:ann}
    \begin{aligned}
         \int_{\R^2} \frac{1}{2} \rho(t) &|u(t)|^2 \varphi_n \dd x  -\int_{\R^2} \frac{1}{2} \rho(s) |u(s)|^2  \varphi_n \dd x +  \nu  \int_s^t \int_{\R^2} |\nabla u|^2  \varphi_n \dd x \dd \tau\\  &  = \nu 
        \int_s^t \int_{\R^2} \frac{1}{2}|u|^2 \Delta  \varphi_n \dd x \dd \tau +\int_s^t \int_{\R^2} \frac{1}{2} \rho |u|^2 u \cdot \nabla  \varphi_n \dd x 
        \dd \tau +\int_s^t \int_{\R^2}  P  u  \cdot \nabla  \varphi_n \dd x 
        \dd \tau.
    \end{aligned}
\end{equation}
Before studying the convergence of \cref{eq:ann}, we recall that, for every $p \in [2,4]$, we have
\begin{align}\label{eq:trivEst}
     \norm{u}^p_{L^p((s,t)\times \R^2)}&  \leq C, 
\end{align}
where $C$ is a constant depending on $p$, $c_0$, $(t-s)$ and $\norm{\sqrt{\rho_0}u_0}_{L^2(\R^2)}$. Indeed, 
\begin{align*}
    \norm{u}^p_{L^p((s,t)\times \R^2)}&  \leq C \int_s^t \norm{\nabla u}_{L^2(\R^2)}^{p-2} \norm{u}_{L^2(\R^2)}^{2} \dd \tau \\ &\leq \frac{C}{c_0}\norm{\sqrt{\rho} u}_{L^\infty ((s,t);L^2(\R^2))}^{2} \int_s^t \norm{\nabla u}_{L^2(\R^2)}^{p-2}  \dd \tau \\ &\leq \frac{C}{c_0}\norm{\sqrt{\rho_0} u_0}_{L^2(\R^2)}^{2} \norm{\nabla u}_{L^2((s,t)\times \R^2)}^{p-2}  (s-t)^{\frac{4-p}{2}} \\ &\leq \frac{C}{c_0}\norm{\sqrt{\rho_0} u_0}_{L^2(\R^2)}^{2} \norm{\sqrt{\rho_0} u_0}_{L^2(\R^2)}^{p-2}  (s-t)^{\frac{4-p}{2}},
\end{align*}
where we used {Gagliardo--Nirenberg's inequality} in the first inequality and the energy inequality in the second and fourth inequality.
The left hand side of \cref{eq:ann} converges to 
 \begin{align*}
                 \int_{\R^2} \frac{1}{2} \rho(t) |u(t)|^2  \dd x-\int_{\R^2} \frac{1}{2} \rho(s) |u(s)|^2   \dd x +   \nu \int_s^t \int_{\R^2} |\nabla u|^2  \dd x \dd \tau
        \end{align*}
    by Lebesgue's dominated convergence theorem, since $\sqrt{\rho} u \in L^\infty((0,\infty); L^2(\R^2))$, we have   $\nabla u \in L^2((0,\infty);L^2(\R^2))$ and $\varphi_n \to 1 $ pointwise. 
 Using \cref{eq:GradLap} and Holder's inequality we can control, up to a constant depending only on $\varphi$, also the first two terms of \cref{eq:ann} by
     \begin{align*}
         \frac{1}{2 n^2 } \norm{u}^2_{L^2((s,t)\times \R^2)} + \frac{C_0}{2 n } \norm{u}^3_{L^3((s,t)\times \R^2)} 
     \end{align*}
     which converges to zero by \cref{eq:trivEst} .
We are left to prove that
\begin{align}\label{eq:pressEst}
    \int_s^t  \int_{\R^2} P u \cdot \nabla  \varphi_n \dd x 
        \dd \tau \to 0.
\end{align}
To this aim we control \cref{eq:pressEst} as follows
\begin{equation*}
    \begin{aligned}
    \left|  \int_s^t  \int_{\R^2} P u \cdot \nabla  \varphi_n \dd x 
        \dd \tau  \right| & =  \left| \int_s^t  \int_{A_n} \left( P - \fint_{B(0,n)} P \right) u \cdot \nabla  \varphi_n \dd x \dd \tau \right| \\ & \leq \norm{ P - \fint_{B(0,n)} P}_{L^2((s,t)\times B(0,n))} \norm{u}_{L^2((s,t)\times A_n)} \norm{\nabla \varphi_n}_{L^\infty(A_n)} \\ & \leq C \frac{1}{n} \norm{ P - \fint_{B(0,n)} P}_{L^2((s,t)\times B(0,n))} \norm{u}_{L^2((s,t)\times A_n)},
    \end{aligned}
\end{equation*}
where, in the first equality, we used  $\operatorname{div} u = 0$ and, in the last inequality, \cref{eq:GradLap}. Owing to  \cref{eq:trivEst}, we can use Lebesgue's  dominated convergence theorem to deduce   $\norm{u}_{L^2((s,t)\times A_n)}  \longrightarrow 0$. Moreover, using the definition of $\bmo$ spaces, we have, for a  fixed time, 
\begin{align*}
    \frac{1}{n}   \norm{ P - \fint_{B(0,n)} P}_{L^2(B(0,n))} & = \left(\frac{1}{n^2}\int_{B(0,n)} 
    \abs{ P - \fint_{B(0,n)} P }^2 \dd x \right)^{1 \slash 2} \\ 
    & \leq C \left(\fint_{B(0,n)} \abs{P-\fint_{B(0,n)} P }^2 \dd x \right)^{1\slash 2}\\
    & = C \norm{P}_{\bmo_2},
\end{align*}
which implies \cref{eq:pressEst} and concludes the proof.

\end{proof}
\begin{remark}
    In the homogeneous case, it is not necessary to assume the additional $\bmo$ bound on the pressure, as the pressure can be recovered from the relation
    \begin{align*}
        \Delta P = \div \div (u \otimes u).
    \end{align*}
    Since, for $d\in \{2,3\}$, every Leray--Hopf velocity $u$ in $[0,\infty) \times \R^d$ satisfies 
    \begin{align*}
        u \otimes u \in L^4((0,\infty);L^{4/d}(\R^d)),
    \end{align*}
    we obtain (using Calderón--Zygmund-type arguments; see \cite[Lemma 5.1]{RobinsonRodrigoSadowski2016}) that
    \begin{align*}
        P \in L^2((0,\infty);L^{2/d}(\R^d)). 
    \end{align*}
    This suffices to show that 
    \begin{align*}
        u \cdot P \in L^1_\loc ((0,\infty);L^1(\R^d)), 
    \end{align*}
    and the last step of the proof can be performed without the additional $\bmo$ bound. This is the main technical difference compared to \cite{Kukavica2006}, where Kukavica demonstrated that the energy equality for Leray–Hopf solutions in three spatial dimensions holds under the weaker assumption
    \[
        P \in L^2_\loc([0,\infty) \times \R^d).
    \]
\end{remark}

We conclude this section by proving \cref{cor:equality}.

\begin{proof}[Proof of \cref{cor:equality}]
The proof follows from an adaptation of the proofs from this section. It is sufficient to use test functions that are not zero at the origin:
\begin{align*}
    \varphi \in C^\infty_c([0,\infty) \times \R^2).
\end{align*}
This adjustment is possible due to the increased rigidity of the assumption
\begin{align*}
    P \in L^2_\loc([0,\infty)\times \R^2) \cap L^2_\loc([0,\infty); \bmo(\R^2)).
\end{align*}
\end{proof}

\subsection{Construction of a suitable pressure}
\label{ssec:pr2}

Finally, in this section, we prove the implication \cref{it:immediately} $\implies$ \cref{it:pressure} of \cref{thm:main2D}.

\begin{proposition}[Implication \cref{it:immediately} $\implies$ \cref{it:pressure}]\label{lem:pressureconstruction}
   Let $(\rho, u)$ be an immediately strong Leray--Hopf solution of \cref{eq:ns}. Then there exists a pressure $P$ associated to $(\rho,u)$  satisfying 
    \begin{align*}
        P \in L^2_\loc((0,\infty) \times \R^2) \cap L^2_\loc((0,\infty); \bmo(\R^2)).
    \end{align*}
\end{proposition}
\begin{proof}
Using the additional bounds of \cref{it:A1A2A3}, which we can use due to the equivalence of \cref{it:immediately} and \cref{it:A1A2A3}, we obtain that
    \begin{align*}
        G \coloneqq - \rho \partial_t u - \rho (u \cdot \nabla) u + \nu \Delta u \in L^4((\eps,
        \infty);L^4(\R^2)) \qquad \text{for every $\eps>0$}.
    \end{align*}
For fixed time $t>0$, we have that $G(t)$ is curl-free. Therefore, we can use \cref{prop:antigradop} to define
 \begin{align*}
        P^\eps \coloneqq \Phi(G^\eps) \in L^4((\eps,\infty);C^{1/2}_0(\R^2)\cap \bmo(\R^2)).
\end{align*}
  By the arbitrariness of $\eps$, we deduce 
    \begin{align*}
        P \in L^2_\loc((0,\infty) \times \R^2)\cap L^2_\loc((0,\infty); \bmo(\R^2)).
    \end{align*}
    Furthermore, using \cref{prop:antigradop} and integrating by parts, we have  
    \begin{align*}
        \int_0^\infty \int_{\R^2} P \dive \varphi \dd x &= 
        - \int_\eps^\infty \int_{\R^2} \left( - \rho \partial_s u - \rho (u \cdot \nabla) u + 
        \nu \Delta u \right) \varphi \dd x \dd s\\
        &= - \int_0^\infty \int_{\R^2} \rho \partial_s \varphi \cdot u+
        \rho u \otimes u : \nabla \varphi \dd x \dd s  
        + \nu \int_0^\infty \int_{\R^2} \nabla u : \nabla \varphi \dd x \dd s,
    \end{align*}
    for every $\varphi \in C_{c}^\infty((0,\infty) \times \R^2;\R^2)$. This completes the proof.
\end{proof}

\section{Application to weak-strong uniqueness}
\label{sec:uniqueness}

In this section, we prove \cref{thm:weakstrong2D}.

\begin{proof}[Proof of \cref{thm:weakstrong2D}]
Let $(\rho_1,u_1)$ and $(\rho_2,u_2)$ be as described in \cref{thm:weakstrong2D}. We use the following notation:
\begin{align*}
    \delta \rho \coloneqq \rho_1 - \rho_2, \quad \delta u \coloneqq u_1 - u_2.
\end{align*}
We divide the proof into several steps.

\textbf{Step 1.} \emph{Admissible test function.} Using the continuity of $u_2$ in 
$t=0$ and the fact that  
\begin{align*}
    A^0_1(u_2), \ A^0_2(u_2), \ A^0_3(u_2) < \infty,
\end{align*}
one can show as in~\cite[Lemma 4.1]{CrinBaratSkondricViolini2024} that, for almost every $t>0$,
\begin{align}\label{eq:energydiff}
    \begin{aligned}
    \frac{1}{2}\norm{\sqrt{\rho_1(t)} \delta u(t)}_{L^2}^2 + \nu \int_0^t 
    \norm{\nabla \delta u}^2_{L^2} \dd s  &\leq -\int_0^t \int_{\R^2} \delta \rho 
    \dot{ u_2} \cdot \delta u +\rho_1 \delta u \otimes \delta u : \nabla u_2 
    \dd x \dd s \\ & \eqqcolon I_1(t) + I_2(t).
    \end{aligned}
\end{align}

\textbf{Step 2.} \emph{Estimate of $I_2$.} Since $\rho_1$ is bounded from below, we get,
for every $t>0,$
\begin{align*}
    \abs{I_2(t)} \leq & \int_0^t \norm{\nabla u_2}_{L^2} \norm{\delta u(s)}^2_{L^4} \dd s\\
    \leq & \int_0^t \norm{\nabla u_2}_{L^2} \norm{\delta u(s)}^2_{L^2} 
    \norm{\nabla \delta u(s)}^2_{L^2} \dd s\\
    \leq & C \int_0^t \norm{\nabla u_2}_{L^2}^2 \norm{\delta u(s)}_{L^2}^2 \dd s + 
    \frac{\nu}{4}\int_0^t \norm{\nabla \delta u(s)}^2_{L^2} \dd s.
\end{align*}

\textbf{Step 3.} \emph{Estimate of $I_1$.} Owing to~\cite[Lemma 4.2]{CrinBaratSkondricViolini2024},
we know that, for every $t>0$,
\begin{align*}
    \abs{I_1(t)} & \leq C \int_0^t \int_0^s \norm{\rho_1 \delta u(\tau)}_{L^4} \dd \tau 
    \norm{\nabla (\dot{u}_2 \cdot \delta u)}_{L^{4/3}} \dd s \\
    & \leq C \int_0^t \int_0^s \norm{\nabla \delta u(\tau)}_{L^2}^{1/2} \dd \tau 
    \sup_{\tau \leq s} \norm{\sqrt{\rho_1(\tau)} \delta u(\tau)}^{1/2}_{L^2}
    \norm{\nabla (\dot{u}_2 \cdot \delta u)}_{L^{4/3}} \dd s.
\end{align*}
Using Ladyzhenskaya's inequality, we get  
\begin{align*}
        \norm{\nabla (\dot{u}_2(s) \cdot \delta u(s))}_{L^{4/3}} \leq &  
    \norm{\nabla \dot{u}_2(s)}_{L^2}\norm{\delta u(s)}_{L^4}
    +\norm{\dot{u}_2(s)}_{L^4}\norm{\nabla \delta u(s)}_{L^2}\\
    \lesssim & \norm{\nabla \dot{u}_2(s)}_{L^2}\norm{\delta u(s)}_{L^2}^{1/2} 
    \norm{\nabla \delta u(s)}_{L^2}^{1/2}
    + \norm{\dot{u}_2(s)}_{L^2}^{1/2}  
    \norm{\nabla \dot{u}_2(s)}_{L^2}^{1/2} \norm{\nabla \delta u(s)}_{L^2}\\
     \lesssim  & \norm{\nabla \dot{u}_2(s)}_{L^2} 
    \sup_{\tau \leq s} \norm{\sqrt{\rho_1(\tau)} \delta u(\tau)}^{1/2}_{L^2}
    \norm{\nabla \delta u(s)}_{L^2}^{1/2}\\
    & + \norm{\dot{u}_2(s)}_{L^2}^{1/2}  
    \norm{\nabla \dot{u}_2(s)}_{L^2}^{1/2} \norm{\nabla \delta u(s)}_{L^2}.
\end{align*}
Collecting the previous estimates in this step, and using Young's inequality with exponents $(4,4,2)$, we conclude
\begin{align*}
    \abs{I_1(t)} \leq & \int_0^t s \norm{\nabla \dot{u}_2(s)}_{L^2} 
    \sup_{\tau \leq s} \norm{\sqrt{\rho_1(\tau)} \delta u(\tau)}_{L^2}  \norm{\nabla \delta u(s)}_{L^2}^{1/2}
    \frac{1}{s}\int_0^s \norm{\nabla \delta u(\tau)}_{L^2}^{1/2} \dd \tau \dd s \\
    & + \int_0^t s \norm{\dot{u}_2(s)}_{L^2}^{1/2}  
    \norm{\nabla \dot{u}_2(s)}_{L^2}^{1/2}
    \sup_{\tau \leq s} \norm{\sqrt{\rho_1(\tau)} \delta u(\tau)}_{L^2}^{1/2} 
    \norm{\nabla \delta u(s)}_{L^2}
    \frac{1}{s}\int_0^s \norm{\nabla \delta u(\tau)}_{L^2}^{1/2} \dd \tau \dd s \\
    \leq & \int_0^t \left(s^2\norm{\nabla \dot{u}_2(s)}_{L^2}^2 
    + s^4\norm{\dot{u}_2(s)}_{L^2}^2 \norm{\nabla \dot{u}_2(s)}_{L^2}^2 \right)
    \sup_{\tau \leq s} \norm{\sqrt{\rho_1(\tau)} \delta u(\tau)}_{L^2}^2 \dd s\\
    & + \frac{\nu}{8} \int_0^t \left( \frac{1}{s}\int_0^s \norm{\nabla \delta u(\tau)}_{L^2}^{1/2} \dd \tau\right)^4 \dd s + 
    \frac{\nu}{8} \int_0^t \norm{\nabla \delta u(s)}^2_{L^2} \dd s.
\end{align*}
Using 
\begin{align*}
    s^2 \norm{\dot{u}_2(s)}_{L^2}^2 \leq C
\end{align*}
and 
\begin{align*}
    \int_0^t \left( \frac{1}{s}\int_0^s \norm{\nabla \delta u(\tau)}_{L^2}^{1/2} \dd \tau\right)^4 \dd s \leq \int_0^t \norm{\nabla \delta u(s)}^2_{L^2} \dd s
\end{align*}
yields
\begin{align*}
    \abs{I_1(t)} \leq & \int_0^t s^2\norm{\nabla \dot{u}_2(s)}_{L^2}^2 
    \sup_{\tau \leq s} \norm{\sqrt{\rho_1(\tau)} \delta u(\tau)}_{L^2}^2 \dd s
    + \frac{\nu}{4} \int_0^t \norm{\nabla \delta u(s)}^2_{L^2} \dd s.
\end{align*}

\textbf{Step 4.} \emph{Conclusion.} Gathering the above estimates, we obtain 
\begin{align*}
   & \frac{1}{2} \sup_{s \leq t} \norm{\sqrt{\rho_1(s)} \delta u(s)}_{L^2}^2 + 
    \frac{\nu}{4} \int_0^t \norm{\nabla \delta u(s)}^2_{L^2} \\ &\leq 
    \int_0^t \left(\norm{\nabla u_2(s)}^2_{L^2} + s^2\norm{\nabla \dot{u}_2(s)}_{L^2}^2 \right)
    \sup_{\tau \leq s} \norm{\sqrt{\rho_1(\tau)} \delta u(\tau)}_{L^2}^2 \dd s.
\end{align*}
Defining
\begin{align*}
    f(s) \coloneqq \sup_{\tau \leq s} \frac{1}{2} \norm{\sqrt{\rho_1(\tau)} \delta u(\tau)}_{L^2}^2,
\end{align*}
we have 
\begin{align*}
    f(t) \leq \int_0^t \left(\norm{\nabla u_2(s)}^2_{L^2} + s^2\norm{\nabla \dot{u}_2(s)}_{L^2}^2 \right) f(s) \dd s.
\end{align*}
Since $f \in L^\infty((0,\infty))$ and 
\begin{align*}
    \norm{\nabla u_2(s)}^2_{L^2} + s^2\norm{\nabla \dot{u}_2(s)}_{L^2}^2 \in L^1((0,\infty)),
\end{align*}
we can use Gr\"onwall's inequality to deduce that 
\begin{align*}
    \delta u \equiv 0.
\end{align*}
Using that $\nabla u_2 \in L^1_\loc([0,\infty);L^\infty(\R^2))$ and 
\begin{align*}
    \partial_t \delta \rho + \dive(u_2 \delta \rho) = \rho_1 \delta u =0, \quad \delta \rho(0) = 0,
\end{align*}
we deduce, for instance from \cite{DiPernaLions1989}, that $\delta \rho  \equiv 0.$
\end{proof}

\appendix
\section{Commutator estimates}
\label{app:conv-comm}
We provide a brief outline of the DiPerna--Lions' commutator estimate from~\cite{DiPernaLions1989}. Let $K_\gamma$ denote the \emph{$\gamma$-neighborhood} of a set $K$ and let $f_\gamma$ be the mollification of $f$ in space, \ie, the convolution of $f$, with respect to the space variable, with the standard mollifier $\eta_\gamma$.

\begin{lemma}\label{lem:commutator1}
    Let $u \in L^1_\loc((0,\infty);H^1_\loc (\R^d))$ and $\rho \in L_\loc^\infty((0,\infty) \times \R^d)$. Then 
    \begin{align*}
    \div ((\rho u)_\gamma - \rho_\gamma u), \quad  
    \frac{(\rho u)_\gamma - \rho_\gamma u}{\gamma}, \quad  
    \nabla((\rho u)_\gamma - \rho_\gamma u) \longrightarrow 0 \quad \text{in\: $L^2_\loc ( (0,\infty) \times \R^d)$ \ 
as $\gamma \to 0$}.
\end{align*}
\end{lemma}

\begin{proof}
Let \( I \subset (0, \infty) \) and \( K \subset \mathbb{R}^d \) be compact sets. Expanding the convolution and using Minkowski's inequality, we have:
\begin{align*}
    \norm{\div \left((\rho u)_\gamma - \rho_\gamma u\right)}_{L^2(I \times K)} &=  \norm{ \int_{B} \rho(\cdot,\cdot+\gamma y) \frac{u(\cdot,\cdot+\gamma y) - u}{\gamma} \cdot \nabla \eta(y) \, \dd y }_{L^2}  \\ 
    &\leq \int_{B} \norm{  \rho(\cdot,\cdot+\gamma y) \frac{u(\cdot,\cdot +\gamma y) - u}{\gamma} \cdot \nabla \eta(y) }_{L^2} \dd y .
\end{align*}
We recall that, by the properties of $L^p$ and Sobolev functions, the following holds:
\begin{align*}
    \norm{\frac{u(t,\cdot+\gamma y)-u(t,\cdot)} {\gamma} - \nabla u(t,\cdot) y}_{L^2(K)} &\longrightarrow 0 \quad \text{as $\gamma \to 0$}
    \end{align*}
    and
    \begin{align*}
    \norm{\rho}_{L^\infty} &\leq C.
\end{align*}
Applying Fatou's lemma, we then obtain
\begin{align}\label{eq:divF}
    \limsup_{\gamma \to 0} \norm{\div \left((\rho u)_\gamma - \rho_\gamma u\right)}_{L^2(I \times K)} \leq  \int_B \norm{\rho\,[\nabla u \cdot y] \cdot \nabla \eta(y)}_{L^2} \dd y.
\end{align}
Similarly, for the other terms, we obtain the following bounds:
\begin{equation}\label{eq:othF}
    \begin{aligned}
        \limsup_{\gamma \to 0} \norm{\frac{(\rho u)_\gamma - \rho_\gamma u}{\gamma}}_{L^2(I \times K)} &\leq  \int_B \norm{\rho\, [\nabla u\cdot y] \  \eta(y)}_{L^2} \dd y, \\
        \limsup_{\gamma \to 0} \norm{\nabla \left((\rho u)_\gamma - \rho_\gamma u\right)}_{L^2(I \times K)} &\leq  \int_B \norm{\rho\, [\nabla u\cdot y] \otimes \nabla \eta(y)}_{L^2} \dd y.
    \end{aligned}
\end{equation}
The integrands on the right-hand sides of \cref{eq:divF} and \cref{eq:othF} exhibit odd symmetry (due to the radial symmetry of the convolution kernel in \( y \)), and since the unit ball has radial symmetry, we conclude that these terms are zero.
\end{proof}

\begin{lemma} \label{lem:deriv-conv}
    Let $f \in L_{\loc}^p(\R^2)$ for $1\leq p \leq \infty$. For every compact set $K \subset \R^2$ we have:
    \begin{align*}
        \gamma \norm{\nabla f_\gamma}_{L^p(K)} \leq \norm{f}_{L^p(K_\gamma)}.
    \end{align*}
\end{lemma}

\begin{proof}
If $f \in L^p(\R^2)$, by Young's convolution inequality, we obtain
\begin{align*}
    \norm{\nabla f_\gamma}_{L^p(\R^2)} \leq \norm{f}_{L^p(\R^2)} \norm{\nabla \eta_\gamma}_{L^p(\R^2)} \leq \frac{1}{\gamma} \norm{f}_{L^p(\R^2)},
\end{align*}
where $\eta$ is the mollification kernel. If $f \in L_\loc^p(\R^2)$ and $K$ is a compact set, we define the extended function $\Bar{f}(x) \coloneqq \mathds{1}_{K_\gamma}(x) f(x)$ and compute 
\begin{align*}
    \norm{\nabla f_\gamma}_{L^p(K)} = \norm{\nabla \Bar{f}_\gamma}_{L^p(K)} \leq \norm{\nabla \Bar{f}_\gamma}_{L^p(\R^2)} \leq \frac{1}{\gamma} \norm{\Bar{f}}_{L^p(\R^2)} = \frac{1}{\gamma} \norm{f}_{L^p(K_\gamma)}.
\end{align*}
\end{proof}

\section{BMO functions}
\label{app:bmo}

In this appendix, we recall some preliminary notions on functions of bounded mean oscillation. Let $\mathcal{Q}$ denote the set of cubes of $\R^d$ with edges parallel to the coordinate axes, and let $\mathcal{B}$ denote the set of balls in $\R^d.$ We say that a locally integrable function $f$ on $\mathbb{R}^d$ has \emph{bounded mean oscillation}, \ie, $f \in \mathrm{BMO}(\R^d)$, if 
\begin{align*}
\sup_{Q \in \mathcal{Q}} \fint_Q\left|f(x)-\fint_Q  f(y) \, \d y \right| \, \d x<\infty.
\end{align*}
By identifying functions that differ only by constants, $\bmo(\R^d)$ becomes a Banach space equipped with the norm 
\begin{align*}
    \norm{f}_{\bmo(\R^d)} = \sup_{Q \in \mathcal{Q}} \fint_Q\left|f(x)-\fint_Q  f(y) \, \d y \right|\d x.
\end{align*}
Equivalently, we could replace the cubes by balls; namely, the following result holds. 
\begin{lemma}\label{lem:bmocubesballs}
    There are constants $0<c<C$ such that
    \begin{align*}
        c \norm{f}_{\bmo(\R^d)} \leq \sup_{B \in \mathcal{B}} \fint_B\left|f(x)-
        \fint_B  f(y) \, \d y \right|\d x \leq C \norm{f}_{\bmo(\R^d)}.
    \end{align*}
\end{lemma}
The following embedding is crucial in our computations.
\begin{lemma}\label{lem:embbmo}
    There exists a constant $C>0$ such that, for every $f \in H^1_{\loc}(\R^2)$ with
    $\nabla f \in L^2(\R^2),$
    \begin{align*}
        \norm{f}_{\bmo(\R^2)} \leq C \norm{\nabla f}_{L^2(\R^2)}.
    \end{align*}
\end{lemma}

\begin{proof}
    According to \cref{lem:bmocubesballs}, it suffices to show that, for every ball $B \in \mathcal{B}$, 
    \begin{align*}
        \fint_B \left|f(x) - \fint_B  f(y) \, \d y \right| \d x
        \leq \int_B \abs{\nabla f}^2 \dd x.
    \end{align*}
     Using Hölder's inequality, we obtain
    \begin{align*}
        \fint_B \left|f(x) - \fint_B  f(y) \, \d y \right| \d x \leq \frac{C}{R} 
        \left(\int_B \left|f(x) - \fint_B  f(y) \, \d y \right|^2 \d x \right)^{1/2}.
    \end{align*}
    We now employ Poincar\'e's inequality for balls, \ie, for every ball $B=B(x_0,R)$, there exists a constant $C$ (independent of the center $x_0$ and the radius $R$) such that, for every $f \in H^1(B)$,
    \begin{align*}
        \left(\int_B \left|f(x) - \fint_B  f(y) \, \d y \right|^2 \d x \right)^{1/2} \leq C R\left(\int_B \abs{\nabla f}^2 \dd x \right)^{1/2}.
    \end{align*}
    This concludes the proof.
\end{proof}
For $f \in L^p_\loc(\R^d)$, with $1 \leq p <\infty$, we define
\begin{equation*}
    \|f\|_{\bmo_p} \coloneqq  \sup_Q \left(  \fint_Q \left| f(x) -  \fint_Q f(y) \,  \mathrm dy \right|^p \, \mathrm d x \right)^{1/p}
\end{equation*}
and set
\begin{equation*}
    \bmo_p \coloneqq  \left\{ f \in L^p_{\text{loc}}(\mathbb{R}^d) : \|f\|_{\bmo_p} < \infty \right\}.
\end{equation*}
We use the following result on the equivalence of these two spaces (for a proof, see~\cite[Corollary 5.2.5]{Auscher2008}).
\begin{lemma}\label{lem:BMop}
    For every $1 < p <\infty$, we have $\bmo = \bmo_p$ and $\|f\|_{\bmo_p} \approx \|f\|_{\bmo}$.
\end{lemma}

\section{Anti-gradient operator}
\label{app:grad}

We start by recalling some notions on H\"older continuous functions. 

\begin{lemma}[H\"older continuous functions] \label{lem:Hold}
With the $\alpha$-Hölder seminorm 
\begin{align*}
    \abs{u}_{\alpha} \coloneqq \sup_{\substack{x, y \in \mathbb{R}^2 \\ x \neq y}} \frac{\abs{u(x) - u(y)}}{\abs{x - y}^{\alpha}} \quad \alpha\in (0,1),
\end{align*}
    the space $C^\alpha_0(\R^2) \coloneqq \{u \in C(\R^2): \, u(0)=0, \ \abs{u}_\alpha < \infty \}$
 is a Banach space.
\end{lemma}
\begin{proof}
    The fact that the semi-norm $ \abs{u}_{\alpha}$ is a norm in $C^\alpha_0(\R^2)$ is straightforward. To prove that $(C^\alpha_0(\R^2),\abs{u}_{\alpha}) $ is complete, consider a Cauchy sequence $\{ u_n\}_{n\in \mathbb N} \subset C^\alpha_0(\R^2) $. For $x\in \R^2 \setminus \{ 0\}$ we have
    \begin{align*}
        \abs{u_n (x) -u_m(x)} & \leq \abs{x}^\alpha\frac{\abs{u_n (x) -u_m(x)}}{\abs{x}^\alpha} \\ & = \abs{x}^\alpha \frac{\abs{ (u_n  -u_m)(x) - (u_n -u_m)(0)}}{\abs{x}^\alpha} \\ & \leq \abs{x}^\alpha \abs{u_n - u_m}_{\alpha},
    \end{align*}
    where we used the fact that $u_n(0)=0$ for every $n\in \N$. Then, for every $x\in \R^2$, the sequence $\{ u_n(x)\}_{n\in \mathbb N} \subset \R $ is Cauchy, hence converging. Finally, we define 
    \begin{align*}
    u(x) \coloneqq 
\begin{cases}
\lim_{n \to \infty} u_n(x), &  \text{if } x \in \R^2 \setminus \{ 0\},\\
0, & \text{if } x =0.
\end{cases}
\end{align*} 
We have that $u \in C^\alpha_0(\R^2)$. Indeed,
\begin{align*}
     \frac{\abs{u(x) - u(y)}}{\abs{x - y}^{\alpha}} = \lim_{n\to \infty}     \frac{\abs{u_n(x) - u_n(y)}}{\abs{x - y}^{\alpha}} \leq C <\infty,
\end{align*}
where the uniform bound holds  since $\{ u_n\}_{n\in \mathbb N}$ is a Cauchy sequence in  $C^\alpha_0(\R^2) $. To conclude the proof, we have to show $\abs{u_n - u}_{C^\alpha} \rightarrow 0$. For every $\eps>0$, let $N$ be large enough to have $\abs{u_n - u_m}_{C^\alpha} < \eps$ for every $n,m > N$. For $n> N$,  we have
\begin{align*}
    \abs{u_n-u}_{C^\alpha} & \leq \sup_{x\neq y}\frac{\abs{(u_n-u)(x)-(u_n-u)(y)}}{\abs{x-y}^\alpha} \\  & = \sup_{x\neq y} \lim_{m \to \infty}\frac{\abs{(u_n-u_m)(x)-(u_n-u_m)(y)}}{\abs{x-y}^\alpha} \\  & \leq \sup_{x\neq y} \limsup_{m \to \infty} \abs{u_n-u_m}_{C^\alpha}\\  & < \eps,
\end{align*}
which concludes the proof.
\end{proof}

Next, we define the Banach space of curl-free functions in $L^p \cap L^2$, with $p >2$, and prove the existence of a unique linear operator that serves as the inverse of the gradient (in a suitable sense).

\begin{proposition}[Anti-gradient operator]\label{prop:antigradop}
    Let $p>2$ and $\alpha_p = 1-\frac{2}{p}$. Let $G^p$ be the Banach space of curl-free $L^p(\R^2; \R^2) \cap L^2(\R^2; \R^2)$ functions:
    \begin{align*}
    G^p(\R^2) \coloneqq \{ g \in L^p(\R^2;\R^2) \cap L^2(\R^2;\R^2):\, \skal{g, \nabla^\perp \psi}=0  
    \text{ for all } \psi \in C^\infty_c \}.
\end{align*}
    There exists a unique linear operator
    $ \Phi \colon G^p(\R^2;\R^2) \to C^{\alpha_p}_0(\R^2;\R) \cap \bmo(\R^2;\R)$
    such that 
    \begin{align*}
        \Phi(g) \in C^\infty(\R^2) \quad \text{and} \quad\nabla \Phi(g) = g, \qquad \text{for all } g \in  G^p(\R^2) \cap C^\infty(\R^2).
    \end{align*}
    Furthermore, for every $g \in G^p(\R^2)$, we have 
    \begin{align*}
        \skal{\Phi(g), \div \psi} = - \skal{g,\psi}, \qquad \text{for all } \psi \in C^\infty_c(\R^2).
    \end{align*}
\end{proposition}

\begin{proof}
    We define $\Phi^r: G^p(\R^2) \cap C^\infty(\R^2) \rightarrow C^\infty(\R^2)$ by setting
    \begin{align*}
        \Phi^r(g)(x) = \int_0^1 g(\tau x) \cdot x \, \mathrm{d}\tau.
    \end{align*}
    Since $\curl g = 0$ for every $g \in G^p \cap C^\infty$, we have $\nabla \Phi^r(g) = g$.
    Furthermore, by Morrey's inequality, we obtain
    \begin{align*}
        \norm{\Phi^r(g)}_{C^{\alpha_p}} \leq C \norm{\nabla \Phi^r(g)}_{L^p} = C \norm{g}_{L^p}.
    \end{align*}
    By \cref{lem:embbmo}, we can also estimate the $\bmo$ norm:
    \begin{align*}
         \norm{\Phi^r(g)}_{\bmo} \leq C \norm{\nabla \Phi^r(g)}_{L^2} = C \norm{g}_{L^2}.
    \end{align*}
    Then $\Phi^r: G^p \cap C^\infty \rightarrow C^{\alpha_p}_0 \cap \bmo$ is a linear and bounded operator between Banach spaces (see \cref{lem:Hold}). Since $G^p(\R^2) \cap C^\infty(\R^2)$ is dense in $G^p(\R^2)$, we can extend uniquely $\Phi^r$ to an operator $\Phi :G^p \rightarrow C^\alpha_0 \cap \bmo$ which satisfies 
    \begin{align*}
        \Phi(g) = \lim_{n \to \infty} \Phi^r(g_n) \quad \text{whenever } \norm{g_n- g}_{L^p} \rightarrow 0.
    \end{align*}
    Let $g \in G^p$ and $\{g_n\}_{n\in \N} \subset G^p \cap C^\infty$ such that $\norm{g_n - g}_{L^p} \rightarrow 0.$ Then, for every $\psi \in C^\infty_c(\R^2)$, we have
    \begin{align*}
        \skal{\Phi(g), \div \psi} = \lim_{n \to \infty} \skal{\Phi^r(g_n), \div \psi} =
        - \lim_{n \to \infty} \skal{\nabla \Phi^r(g_n), \psi} = 
        - \lim_{n \to \infty} \skal{g_n, \psi} = -\skal{g,\psi}.
    \end{align*}
    It remains to prove the uniqueness of the operator. To this end, we let $\Tilde{\Phi} \colon
    G^p \to C^\alpha_0 \cap \bmo$ be another operator with the same properties as $\Phi$. For every $g \in G^p \cap C^\infty$, we have
    \begin{align*}
        \nabla(\Tilde{\Phi}(g) - \Phi(g)) = g - g =0.
    \end{align*}
    Hence, there exists a constant $c \in \R$ such that $\Tilde{\Phi}(g) = \Phi(g) + c$. From 
    this, we deduce
    \begin{align*}
        0 = \Tilde{\Phi}(g)(0) = \Phi(g)(0) + c = c,
    \end{align*}
    and, therefore, we have $\Tilde{\Phi}(g) = \Phi(g)$ for every $g \in G^p \cap C^\infty$. 
    Since the extension to $G^p$ is unique, we conclude that $\Phi = \Tilde{\Phi}$.
\end{proof}

\vspace{0.5cm}
\section*{Acknowledgments}

T.~Crin-Barat is supported by the project ANR-24-CE40-3260,  ``Hyperbolic Equations, Approximations \& Dynamics'' (HEAD). 

N.~De~Nitti is a member of the Gruppo Nazionale per l'Analisi Matematica, la Probabilità e le loro Applicazioni (GNAMPA) of the Istituto Nazionale di Alta Matematica (INdAM). He has been partially supported by the Swiss State Secretariat for Education, Research and Innovation (SERI) under contract number MB22.00034 through the project TENSE.

S.~\v{S}kondri\'c and A.~Violini are supported  by the Deutsche Forschungsgemeinschaft (DFG) via the project ``Inhomogeneous and compressible fluids: statistical solutions and dissipative anomalies'' within the SPP 2410 ``Hyperbolic Balance Laws in Fluid Mechanics: Complexity, Scales, Randomness'' (CoScaRa).

This work was partially carried out during A.~Violini's visit to the Chair of Analysis at the Friedrich-Alexander-Universität Erlangen-N\"urnberg (FAU).

We thank G.~Crippa, R.~Danchin, L.~Galeati, and E.~Wiedemann for some helpful conversations.

\vspace{0.5cm}

\printbibliography

\vfill 

\end{document}